\documentclass[11pt,reqno]{amsart}

\usepackage[T1]{fontenc}
\usepackage{mathptmx}
\usepackage{microtype}
\usepackage{amsmath,amssymb,amsthm,mathtools,mathrsfs}
\usepackage{graphicx}
\usepackage{enumitem}
\usepackage[colorlinks=true,linkcolor=blue,citecolor=blue,urlcolor=blue]{hyperref}
\usepackage[nameinlink,capitalise]{cleveref}
\usepackage{geometry}
\numberwithin{equation}{subsection}

\newtheorem{theorem}{Theorem}[subsection]
\newtheorem{proposition}[theorem]{Proposition}
\newtheorem{lemma}[theorem]{Lemma}
\newtheorem{corollary}[theorem]{Corollary}
\newtheorem{problem}[theorem]{Problem}
\crefname{problem}{problem}{problems}
\Crefname{problem}{Problem}{Problems}
\theoremstyle{definition}

\theoremstyle{remark}
\newtheorem{remark}[theorem]{Remark}

\newcommand{\ddc}{\mathrm{dd}^{c}}
\newcommand{\cO}{\mathcal O}

\newcommand{\bP}{\mathbb P}
\newcommand{\C}{\mathbb C}
\newcommand{\R}{\mathbb R}

\newcommand{\vol}{\operatorname{vol}}
\newcommand{\dist}{\operatorname{dist}}

\newcommand{\Id}{\operatorname{Id}}
\newcommand{\Xsurf}{\mathscr{X}}
\newcommand{\Wset}{\mathscr{W}}
\newcommand{\pmeas}{\sigma}
\newcommand{\diff}{\mathrm{d}}
\newcommand{\imag}{\mathrm{i}}
\newcommand{\e}{\mathrm{e}}
\newcommand{\geo}{\mathrm{g}}
\newcommand{\heo}{\mathrm{h}}
\newcommand{\calB}{\mathcal{B}}

\title[Uniform Green packing]{Uniform Green packing and quantitative transversality on compact Riemann surfaces}
\date{}

\author{Haakan Hedenmalm}
\address{
Department of Mathematics
and Computer Science, St Petersburg State University, St Petersburg, Russia.
\\
Beijing Institute for Mathematical Sciences and Applications,
Huairou District, Beijing 101408, China
\\
Department of Mathematics and Statistics, University of Reading, U. K.
\\
Department of Mathematics, University of Aveiro, Portugal.
}
\email{haakan00@gmail.com}

\author{Song-Yan Xie}
\address{State Key Laboratory of Mathematical Sciences, Academy of Mathematics
and Systems Science, Chinese Academy of Sciences, Beijing 100190, China;
School of Mathematical Sciences, University of Chinese Academy of Sciences,
Beijing 100049, China}
\email{xiesongyan@amss.ac.cn}


\begin{document}

\begin{abstract}
Let $\Xsurf$ be a compact Riemann surface equipped with a smooth positive
probability form $\pmeas$, and let $U_{\Xsurf}$ denote its normalized Green
kernel. We show that, for every positive integer $N$, there exists a
configuration $\Wset_N$ of $N$ points such that
\[
 \sup_{z\in\Xsurf}\sum_{w\in\Wset_N}U_{\Xsurf}(z,w)
 \le C(\Xsurf,\pmeas),
\]
where the constant is independent of $N$. This extends uniform Green packing
from the Riemann sphere and complex tori to arbitrary compact Riemann surfaces.

We give two independent proofs and a direct construction of the quantitative transversality input. The first uses quantitatively transverse
holomorphic sections of positive line bundles, while the second relies on
spectral designs and localization properties of the Green kernel. Both
constructions yield optimal-order equidistribution in the Wasserstein distance
$W_1$.

We also establish a uniform positive lower bound for the optimal packing
constant. For the associated $\beta$-Bergman energy, we obtain matching
quadratic upper and lower bounds for its deviation from the Jensen value $1$,
uniformly in $N$ as $\beta$ tends to zero.
\end{abstract}

\maketitle
\setcounter{tocdepth}{1}

\tableofcontents

\section{Introduction}\label{sec:intro}

\subsection{Background and motivation}\label{subsec:intro-background}

For a polynomial of degree $N$, factorization gives
\[
 P(z)=a\prod_{j=1}^{N}(z-w_j),
\]
where $a\ne0$ is its leading coefficient and $w_1,\ldots,w_N$ are its
zeros, counted with multiplicity.  Taking logarithms of absolute values,
we obtain
\[
 \log|P(z)|=\log|a|+\sum_{j=1}^{N}\log|z-w_j|.
\]
Thus a problem concerning the size of polynomial products becomes a problem
concerning sums of logarithmic potentials.  On the Riemann sphere, with the
Fubini--Study metric normalized to have total area one, this observation
leads to geometric Bombieri-type inequalities.  On a flat torus, analogous
pseudopolynomial products can be expressed through Weierstrass sigma
functions, with the automorphy correction required by the lattice quotient.
The corresponding sphere and torus results are proved in
\cite[Theorems~2.1.1 and~2.3.1]{EtayoHedenmalmOrtega}, whose Section~2.2
raises the question of extending the bounds to compact Riemann surfaces.

These two examples suggest that the essential object is not the product
formula itself, but the logarithmic potential it represents.  We therefore
seek an intrinsic formulation that makes sense on an arbitrary compact
Riemann surface, equipped with a smooth positive probability form.

On a general compact Riemann surface $\Xsurf$ with a smooth positive probability form
$\pmeas$,
an effective divisor $D=w_1+\cdots+w_N$ still determines a
canonical divisor section $s_D\in H^0(\Xsurf,\cO(D))$, where
$H^0(\Xsurf,\cO(D))$ is the space of holomorphic sections of the divisor line
bundle $\cO(D)$ associated with $D$.  What is generally absent is
a preferred single-valued scalar product representation analogous to the
affine polynomial factorization.  The normalized Green kernel provides the
intrinsic replacement.  Relative to $\pmeas$, its logarithmic singularity at
$w$
models $\log|z-w|$, while the mean-zero condition fixes its additive constant
globally.  More precisely, pick a Hermitian metric $\heo_D$ on $\cO(D)$ with
first Chern form $c_1(\cO(D),\heo_D)=N\pmeas$.  The Poincar\'e-Lelong formula
and the mean-zero normalization give
\[
 \sum_{j=1}^{N}U_{\Xsurf}(\,\cdot\,,w_j)
 =
 \log|s_D|_{\heo_D}
 -
 \int_{\Xsurf}\log|s_D|_{\heo_D}\,\diff\pmeas.
\]
This naturally leads to a one-sided packing problem. As each Green kernel
tends to $-\infty$ at its pole, no uniform two-sided pointwise estimate is
possible.  The natural question is instead whether the poles can be chosen so
that
\[
 \sup_{z\in \Xsurf}\sum_{j=1}^{N}U_{\Xsurf}(z,w_j)
\]
remains bounded independently of $N$.  Such an upper bound controls the
positive excursions of the potential; after exponentiation, the negative
logarithmic singularities become zeros and one obtains uniform
$L^\beta$-type product estimates.

\begin{figure}[tbp]
\centering
\includegraphics[width=0.78\textwidth]{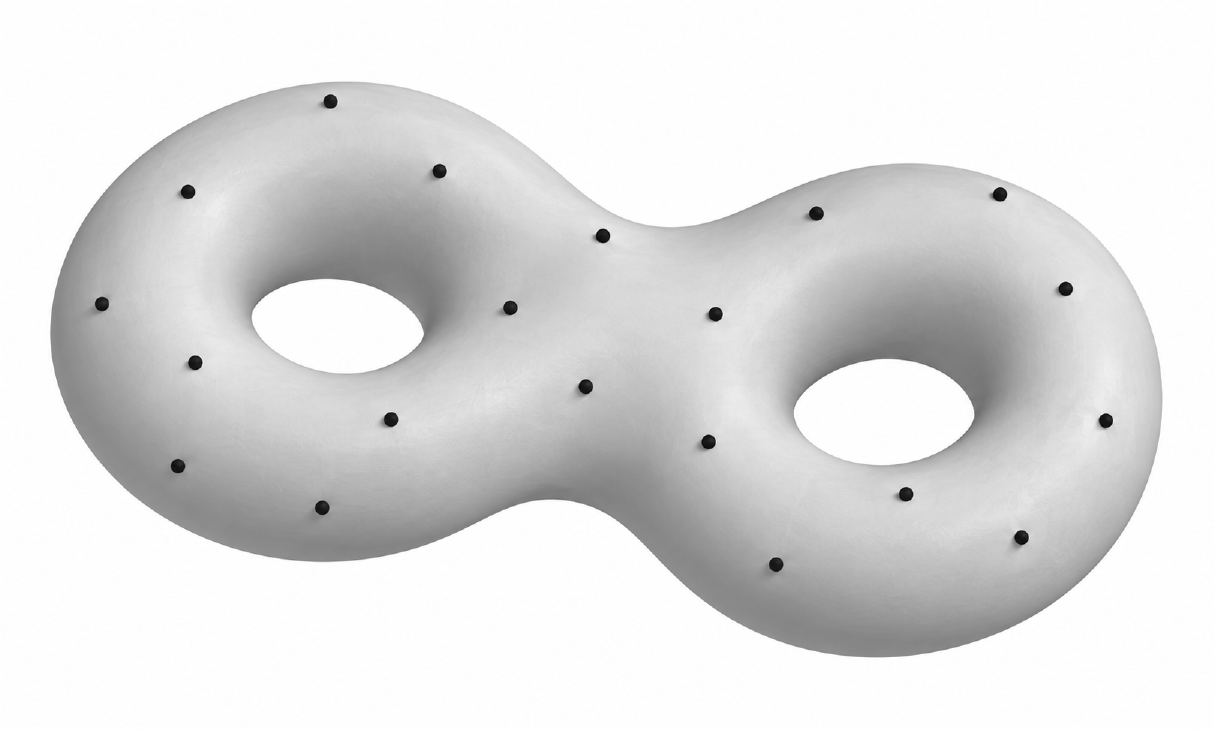}
\caption{A schematic configuration of twenty points on a compact Riemann surface of genus two.}
\label{fig:surface-points}
\end{figure}

We show that this upper bound statement holds generally by two separate
mechanisms, which we denote by Route~A and Route~B.

\begin{figure}[tbp]
\centering
\includegraphics[width=0.6642\textwidth]{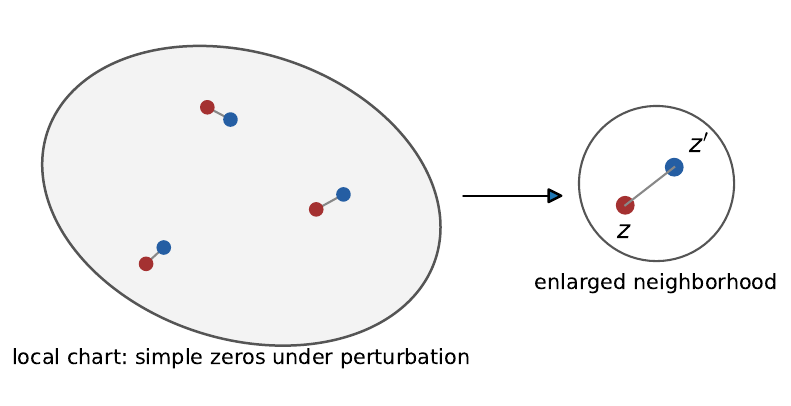}
\caption{Schematic local effect of perturbing a holomorphic section. The red zeros move to the blue zeros while remaining simple under quantitative transversality. The enlarged disk depicts the displacement of one zero in a local coordinate.}
\label{fig:transversality-illustration}
\end{figure}

\emph{Route~A} is holomorphic.  It chooses the points $w_j$ as the zero divisor
of a holomorphic section of a high tensor power of a
positive degree-one line bundle. Uniform rescaled $C^2$ bounds together
with quantitative zero-section transversality yield a uniform local
linear model and separated
simple zeros, while the Poincar\'e-Lelong formula identifies the mean-zero
logarithmic norm with the Green potential.

\emph{Route~B} is spectral, and works directly with points.  Laplace
designs cancel
the spectral modes below the natural frequency scale $\Lambda\asymp\sqrt N$,
and kernel localization together with local counting controls the remaining
high-frequency part.

The routes are genuinely different, and their additional conclusions are
complementary: the first gives local divisor geometry, while the second reveals
a purely spectral and potential-theoretic origin for the packing estimate.

\subsection{Setting and conventions}\label{subsec:conventions}

Let $\Xsurf$ be a connected compact Riemann surface equipped with a smooth
positive $(1,1)$-form $\pmeas$, normalized by
\[
\int_{\Xsurf}\diff\pmeas=1.
\]
We use $\pmeas$ for both this form and its associated measure.  In integrals
$\diff\pmeas$
refers to the measure; in current identities such as $\ddc U=\delta-\pmeas$ it
denotes the form.

The following conventions are used throughout and recorded here to prevent
sign or scale ambiguity in the covering argument of
\Cref{sec:zero-geometry} or in the spectral construction of
\Cref{sec:route-b}.

The normalization of the differential operators fixes the signs and constants in the Green and curvature identities.

\begin{remark}\label{rem:conventions}
We use
\[
\diff_z^c=\frac{\imag}{2\pi}(\bar\partial_z-\partial_z),
\qquad
\ddc_z=\frac{\imag}{\pi}\partial_z\bar\partial_z,
\]
so that $\ddc_z\log|z|=\delta_0$ on $\C$.  The normalized Green kernel is the
unique symmetric function $U_{\Xsurf}(z,w)$ satisfying
\begin{equation}\label{eq:green-normalization}
 \ddc_zU_{\Xsurf}(z,w)=\delta_w-\pmeas,
 \qquad
 \int_{\Xsurf} U_{\Xsurf}(z,w)\,\diff\pmeas(z)=0.
\end{equation}
It is smooth away from the diagonal.
For the conformal metric whose area form is $\pmeas$, our Laplace--Beltrami
operator $\Delta$ is nonnegative and satisfies
\[
 \ddc f=-\frac1{2\pi}(\Delta f)\,\pmeas .
\]
Consequently, for a real orthonormal eigenbasis $\{\varphi_{\nu}\}_{\nu}$ in
$L^2(\Xsurf,\pmeas)$ with $\Delta\varphi_{\nu}=\lambda_{\nu}^2\varphi_{\nu}$,
where the nonnegative frequencies $\lambda_{\nu}\in\R_{\geq0}$ are listed in
nondecreasing order (thus $\lambda_{\nu}^2$ are the Laplacian eigenvalues),
the Green kernel has the spectral expansion
\[
 U_{\Xsurf}(z,w)=-2\pi\sum_{\nu:\lambda_{\nu}>0}
 \frac{\varphi_{\nu}(z)\varphi_{\nu}(w)}{\lambda_{\nu}^2}
\]
in the distributional sense. For a smooth function $\phi$, the norm of its
two-form $\ddc\phi$ below means
\[
 \|\ddc\phi\|_{\infty}
 :=\left\|\frac{\ddc\phi}{\pmeas}\right\|_{L^\infty(\Xsurf,\pmeas)}.
\]
This convention fixes the numerical constants in the discrepancy estimates.

Route~A uses the rescaled metric $\geo_N=N\geo$ for a fixed compatible
metric $\geo$.  Each covariant derivative measured in $\geo_N$ carries a
factor $N^{-1/2}$ relative to $\geo$.  In particular,
$\dist_{\geo_N}(z,w)=\sqrt N\,\dist_{\geo}(z,w)$.
We indicate the metric explicitly when rescaling changes a derivative
or a distance.  Finally, one symbol serves
two purposes in the literature and must be disambiguated here: $L$ always
denotes the degree-one line bundle of Route~A,
and the spectral cutoff of Route~B is always written $\Lambda$.
\end{remark}

Thus, in a local coordinate at $w$,
\begin{equation}\label{eq:green-local}
 U_{\Xsurf}(z,w)=\log|z-w|+O(1).
\end{equation}
The joint version of this expansion, with a smooth regular part in both
variables, is recorded in \eqref{eq:green-joint-local}.  Unless a dependence
is displayed explicitly, constants may depend on the fixed geometric data
$(\Xsurf,\pmeas,\geo,L,\heo)$, and on $\beta$ in exponential-moment
statements, but never
on $N$.  Subscripts such as $\asymp_{\Xsurf}$ abbreviate this fixed-data
dependence.

\subsection{Further comments on the setting}\label{subsec:setting-comments}

The geometric scale is worth emphasizing before introducing the packing
functional.  For $N$ points distributed over a surface of total area one,
the natural area per point is of order $N^{-1}$, and the corresponding
length scale is $N^{-1/2}$.  The rescaled metric $\geo_N=N\geo$ turns
this microscopic length into a distance of order one.  This is the scale
used both for the zeros of holomorphic sections in Route~A and for the
localized kernels in Route~B.  The metric scaling and the distinction
between the line bundle $L$ and the spectral cutoff $\Lambda$ have been
fixed in the preceding conventions.

The potential-theoretic formulation also clarifies the role of the
exponent $\beta$.  We first choose the points and form their normalized
Green potential.  Exponentiation then produces a nonnegative function
whose geometric mean is fixed.  Its positive moments measure how much
this function can concentrate above its geometric mean.  This motivates
the following energy and its optimal value.

For $\Wset=(w_1,\ldots,w_N)\in \Xsurf^N$ and $\beta>0$, define the \emph{$\beta$-Bergman energy}
\begin{equation}\label{eq:B-def}
\calB_{\beta}(\Wset;\Xsurf):=
\int_{\Xsurf}\exp\Big(\beta\sum_{j=1}^{N}U_{\Xsurf}(z,w_j)\Big)\diff\pmeas(z)
\end{equation}
and
\begin{equation}\label{eq:theta-def}
 \Theta_{N,\beta}(\Xsurf):=\inf_{\Wset\in \Xsurf^N}\calB_{\beta}(\Wset;\Xsurf).
\end{equation}
The logarithmic singularity in \eqref{eq:green-local} tends to $-\infty$, so
the integrand in \eqref{eq:B-def} vanishes, rather than blows up, at the points
$w_j$. Moreover, the normalization of $U_{\Xsurf}$ gives
\[
 \int_{\Xsurf}\sum_{j=1}^NU_{\Xsurf}(z,w_j)\,\diff\pmeas(z)=0.
\]
Consequently, Jensen's inequality yields the lower bound
$\calB_{\beta}(\Wset;\Xsurf)\geq1$ for every configuration $\Wset$.
The substantive issue is whether one can choose $\Wset$ so that
the reverse inequality holds with a constant independent of $N$.  Equivalently,
if
\[
 Y_{\Wset}(z):=\exp\!\Big(\sum_{j=1}^NU_{\Xsurf}(z,w_j)\Big),
\]
then the mean-zero normalization says that the geometric mean of $Y_{\Wset}$
is $1$, whereas the $\beta$-Bergman energy is
$\calB_{\beta}(\Wset;\Xsurf)=\|Y_{\Wset}\|_{L^\beta}^\beta$.  The theorem
therefore asks for products with fixed geometric mean and uniformly bounded
$L^\beta$ moment.  On the sphere these are normalized polynomial products,
whereas on a general Riemann surface the Green kernel replaces the missing
global first degree factors.  This formulation expresses the packing question
intrinsically, independently of the holomorphic and spectral constructions
used below.

\subsection{Statement of the main result}\label{subsec:main-statement}

The lower bound for the energy follows directly from normalization.
The geometric question is whether configurations can be chosen with a
uniform upper bound for the underlying Green potential.  Such a bound is
stronger than controlling any one moment: the same configuration then
controls all positive exponents $\beta$.

The pointwise bound answers the geometric packing question and, through the exponential integral, controls all positive moments using the same configurations.

\begin{theorem}\label{thm:main}
For each compact Riemann surface $(\Xsurf,\pmeas)$ there are a constant
$C_{\Xsurf,\pmeas}<\infty$ and configurations $\Wset_N\in\Xsurf^N$, one for each
$N=1,2,3,\ldots$, such that
\begin{equation}\label{eq:main-pointwise-packing}
 \sup_{z\in \Xsurf}\sum_{w\in \Wset_N}U_{\Xsurf}(z,w)\leq C_{\Xsurf,\pmeas}.
\end{equation}
As a consequence, for each $\beta>0$ and every $N=1,2,3,\ldots$,
\[
1\leq\Theta_{N,\beta}(\Xsurf)\leq \exp(\beta C_{\Xsurf,\pmeas}).
\]
\end{theorem}

The uniform upper-bound constant may depend on the fixed surface and the
probability area form, but not on the number of points or the exponent
$\beta$.  Hence, a single family of configurations works
simultaneously for every $\beta>0$, and the theorem preserves the degree-free
feature of Bombieri's spherical inequality while replacing linear factors
by intrinsic Green factors on a compact Riemann surface of arbitrary genus.

Thus the same exponent-free one-sided Green-potential bound follows from two
genuinely different mechanisms, one holomorphic and one spectral.

\subsection{The two constructions}\label{subsec:two-routes}

The two constructions use the microscopic scale introduced above.  In
Route~A, the metric $\geo_N=N\geo$ gives uniform size to the local geometry of
sections of $L^N$.  In Route~B, the spectral cutoff $\Lambda\asymp\sqrt N$
corresponds to localization at the same length scale $N^{-1/2}$.

Both constructions apply for all sufficiently large $N$, with thresholds
depending only on the fixed geometric data.  The finitely many remaining
degrees are accommodated by enlarging the packing constant.

\emph{Route A.}  We choose a degree-one positive line bundle $(L,\heo)$ with
curvature $\pmeas$ and use holomorphic sections of $L^N$ which are
quantitatively transverse to the zero section at the rescaled metric
$\geo_N$.  If
$s_N\in H^0(\Xsurf,L^N)$ is uniformly bounded in rescaled $C^2$ norm an 
\begin{equation}\label{eq:intro-transverse}
|s_N(x)|_{\heo^N}+|\nabla s_N(x)|_{\geo_N,\heo^N}\geq\eta>0,
\qquad x\in \Xsurf,
\end{equation}
then its $N$ zeros are simple and uniformly separated in $\geo_N$.  Near a zero
$x_0$, the norm of $s_N(x)$ is comparable to $\mathrm{dist}_{\geo_N}(x,x_0)$,
while away from all zeros it has a uniform positive lower bound.  The curvature
equation for $\log|s_N|$ then prevents a zero-free $\geo_N$-ball of arbitrarily
large fixed radius, so the same zeros also have uniformly bounded covering
radius.  We call \eqref{eq:intro-transverse}
\emph{quantitative zero-section transversality},
or equivalently, uniform zero-section nondegeneracy.  In the general
Donaldson-Auroux framework, transversality may be imposed relative to an
auxiliary submanifold, while for the present Riemann surface the relevant
submanifold is simply the zero section of $L^N$, so
\eqref{eq:intro-transverse} is precisely the specialization used throughout.

Three facts now combine: the local linear model, the separation, and the
uniform lower bound away from the zero neighborhoods.  Together they give
\begin{equation}\label{eq:intro-log}
 \int_{\Xsurf}\left|\log|s_N|_{\heo^N}\right|\diff\pmeas\leq C.
\end{equation}
The Poincar\'e-Lelong formula identifies the Green potential of the
zero divisor with the mean-zero normalization of $\log|s_N|_{\heo^N}$:
\begin{equation}\label{eq:intro-potential-identity}
 \sum_{w\in Z_{s_N}}U_{\Xsurf}(z,w)
 =\log|s_N(z)|_{\heo^N}
  -\int_{\Xsurf}\log|s_N|_{\heo^N}\,\diff\pmeas.
\end{equation}
Indeed, the $N$ disjoint zero neighborhoods each contribute $O(N^{-1})$
to the negative logarithmic integral.  Away from these neighborhoods,
transversality supplies a uniform lower bound, while the global $C^0$
estimate controls the positive part.  This proves \eqref{eq:intro-log},
and hence the packing estimate through \eqref{eq:intro-potential-identity}.
The core of Route A is therefore the chain
\[
 \begin{aligned}
  &\text{zero-section transversality}\\[-2pt]
  &\qquad\Longrightarrow \text{separated simple zeros}\\[-2pt]
  &\qquad\Longrightarrow \text{uniform logarithmic integrability}\\[-2pt]
  &\qquad\Longrightarrow \text{Green packing}.
 \end{aligned}
\]

\noindent{\sc Route B.}  In Section \ref{sec:route-b}, we give an independent
potential-theoretic proof with three steps.  First, an optimal equal-weight
design cancels all Laplace eigenfunctions up to the natural frequency
$\Lambda\asymp\sqrt N$.  Next, spectral localization makes the high-frequency
Green remainder decay rapidly beyond distance $\Lambda^{-1}$.  Finally, local
counting of the design nodes sums that remainder uniformly.  The matching of
scales is essential: Weyl's law predicts $O(\Lambda^2)$ low frequencies on a
surface, so an $N$-point design can cancel frequencies only up to
$\Lambda\asymp\sqrt N$.  The remainder is localized at spatial scale
$\Lambda^{-1}\asymp N^{-1/2}$, exactly the point-spacing scale.  Polynomial
local counting then sums the rapidly decaying off-diagonal tail, whereas the
negative logarithmic singularity controls the nearest nodes from above.  Only a
pointwise upper bound for the remainder potential is needed: its negative
logarithmic spikes only make the exponential smaller.  This one-sided nature
is why local counting plus spectral cancellation suffices, even though
Route~B does not provide the two-sided norm model obtained from a transverse
holomorphic section.
Route~B therefore proves the same uniform bound without producing a holomorphic
section and without invoking quantitative transversality at all.

\subsection{Geometry of the two constructions}\label{subsec:intro-geometry}

The two routes carry different additional geometry, recorded in the next
theorem.  We write $\mathfrak g=\operatorname{genus}(\Xsurf)$, fix a smooth
compatible metric $\geo$, write $\geo_N=N\geo$, and associate to an ordered
configuration $\Wset=(w_1,\ldots,w_N)\in \Xsurf^N$ the quantities
\[
D_{\Wset}:=w_1+\cdots+w_N,
\qquad
\nu_{\Wset}:=\frac1N\sum_{j=1}^N\delta_{w_j},
\qquad
\mathcal R_{\geo}(\Wset):=\sup_{x\in\Xsurf}\min_{1\leq j\leq N}
\mathrm{dist}_{\geo}(x,w_j).
\]
Given an $N$-tuple $\Wset=(w_1,\ldots,w_N)\in\Xsurf^N$, the closed
\emph{Voronoi cell} associated with $w_k$ is
\[
\mathcal V_k:=\bigl\{z\in\Xsurf:
\mathrm{dist}_{\geo}(z,w_k)\leq\mathrm{dist}_{\geo}(z,w_j)
\text{ for every }j\bigr\}.
\]
The Voronoi cells are unchanged under the rescaling $\geo_N=N\geo$, since
$\mathrm{dist}_{\geo_N}=\sqrt N\,\mathrm{dist}_{\geo}$.

For two Borel probability measures $\rho_0,\rho_1$ on
$(\Xsurf,\mathrm{dist}_{\geo})$ and $1\leq p<\infty$, their
$p$-Wasserstein distance is
\[
W_p(\rho_0,\rho_1)^p:=\inf_{\pi\in\Pi(\rho_0,\rho_1)}
\int_{\Xsurf\times\Xsurf}\mathrm{dist}_{\geo}(x,y)^p
\diff\pi(x,y),
\]
where $\Pi(\rho_0,\rho_1)$ consists of the Borel probability measures on
$\Xsurf\times\Xsurf$ with marginals $\rho_0$ and $\rho_1$; equivalently,
\[
\int_{\Xsurf\times\Xsurf} f(x)\diff\pi(x,y)=\int_{\Xsurf} f\diff\rho_0,\qquad
\int_{\Xsurf\times\Xsurf} q(y)\diff\pi(x,y)=\int_{\Xsurf} q\diff\rho_1
\]
for all continuous $f,q$.  For $p=\infty$, we use
\[
 W_{\infty}(\rho_0,\rho_1)
 :=\inf_{\pi\in\Pi(\rho_0,\rho_1)}
 \|\mathrm{dist}_{\geo}\|_{L^\infty(\pi)}.
\]

The two constructions also give quantitative information about the distribution of the points.  The holomorphic construction yields separation and covering, while the spectral construction provides local counting and transport bounds.

\begin{theorem}
\label{thm:intro-geometry}
All constants below are independent of $N$.
\begin{enumerate}
\item[\rm (i)] {\sc Route A.}  There are constants
$r_0,R_{\mathrm{cov}},c_V,C_V>0$ and, for all sufficiently large $N$,
sections $s_N\in H^0(\Xsurf,L^N)$ whose divisors
$D_N=Z_{s_N}=w_{1,N}+\cdots+w_{N,N}$ satisfy
\[
 2r_0\leq\min_{i\neq j}\mathrm{dist}_{\geo_N}(w_{i,N},w_{j,N}),
 \qquad
 \sup_{x\in \Xsurf}\mathrm{dist}_{\geo_N}(x,D_N)\leq R_{\mathrm{cov}}.
\]
Moreover, if $\mathcal V_{j,N}$ is the closed $\geo_N$-Voronoi cell of $w_{j,N}$,
then
\[
B_{\geo_N}(w_{j,N},r_0)\subset\mathcal V_{j,N}
\subset B_{\geo_N}(w_{j,N},R_{\mathrm{cov}}),
 \qquad
 \frac{c_V}{N}\leq\pmeas(\mathcal V_{j,N})\leq\frac{C_V}{N}.
\]
In particular, we have the diameter bound
$\operatorname{diam}_{\geo_N}\mathcal V_{j,N}\leq2R_{\mathrm{cov}}$.
\item[\rm(ii)] {\sc Route B.}  
For all sufficiently
large $N$, there are configurations $\tilde \Wset_N\in \Xsurf^N$ with
\[
 \mathcal R_{\geo}(\tilde \Wset_N)\leq C_{\Xsurf} N^{-1/2},
 \qquad
 \#\bigl(\tilde \Wset_N\cap B_{\geo}(x,RN^{-1/2})\bigr)\leq C_{\Xsurf}(1+R)^2
\]
for every $x\in\Xsurf$ and $R\geq0$.  No separation assertion is made for the
spectral nodes, but their multiplicities are uniformly bounded: equivalently,
every atom of their empirical measure has mass at most $C_{\Xsurf}/N$.
Moreover the partition-to-node coupling gives directly
\[
 c_{\Xsurf} N^{-1/2}\leq W_p(\nu_{\tilde \Wset_N},\pmeas)\leq C_{\Xsurf} N^{-1/2},
 \qquad 1\leq p\leq\infty.
\]
\end{enumerate}
\end{theorem}

The theorem collects results proved below.  In part~(i), the existence of the
sections $s_N$ follows from Theorem \ref{thm:curve-transversality}, separation
of their zeros from Proposition \ref{prop:separation}, and the covering-radius
bound from Theorem \ref{thm:transverse-covering}, which relies on Proposition
\ref{prop:norm-model}.  The Voronoi-cell estimates are given in Corollary
\ref{cor:transverse-voronoi}.  Part~(ii) combines Lemma
\ref{lem:partition-compatible-design} with
\eqref{eq:design-local-distance}, \eqref{eq:design-local-count}, and
\eqref{eq:design-covering-radius}; bounded
multiplicity is Corollary \ref{cor:design-bounded-multiplicity}, and the
Wasserstein bounds are proved in Corollary
\ref{cor:design-direct-transport}.

Both constructions (Routes A and B) produce optimal-order equidistribution.
We record only the immediate consequence here: the systematic theory of
what a bounded $\beta$-exponential Green energy entails is the subject of the
companion paper
\cite{PaperB}.

The resulting uniform energy bound gives the following equidistribution estimate, independently of which construction supplies the configurations.

\begin{corollary}
\label{cor:intro-equidistribution}
Let $\Wset_N\in \Xsurf^N$ satisfy
\[
 \sup_N \calB_{\beta}(\Wset_N;\Xsurf)\leq K<+\infty
\]
for a fixed $\beta>0$. 
Then
\[
\left|\int_{\Xsurf}\phi\,\diff(\nu_{\Wset_N}-\pmeas)\right|
\leq\frac{2(1+\log K)}{\beta N}\|\ddc\phi\|_{\infty},
\qquad \phi\in C^\infty(\Xsurf),
\]
and
\[
c_{\Xsurf} N^{-1/2}\leq W_1(\nu_{\Wset_N},\pmeas)\leq C_{\Xsurf,K,\beta}N^{-1/2}.
\]
The exponent is optimal: every probability measure supported on at most $N$
points satisfies $W_1(\nu,\pmeas)\geq c_{\Xsurf} N^{-1/2}$.
\end{corollary}

\subsection{Status of quantitative transversality}
\label{subsec:transversality-status}

The main argument uses the standard integrable quantitative-transversality
theorem stated in \Cref{sec:transversality}, following
\cite{Donaldson,Auroux,MohsenAFST,MohsenLimit}.  This keeps the Green-packing
deduction separate from the technical construction of transverse sections, and
\Cref{prop:transversality-to-packing} isolates exactly which properties of the
holomorphic construction are used.  A curve-specific derivation of the same
properties is provided in \Cref{sec:direct-transversality-appendix}, independently of the
general theorem and of Bergman-kernel asymptotics; its sole analytic existence
theorem is the weighted $L^2$ result for $\bar\partial$.  Nothing here
depends on which of the two derivations the reader prefers.

\subsection{Relation to the companion papers}\label{subsec:companions}

The material of this paper was originally part of a single long manuscript
which also developed the consequences of a bounded $\beta$-exponential Green
energy and a condition theory for zero divisors.  These developments are
presented separately.  The consequences of the bounded-energy hypothesis
are independent of the constructions given here: \cite{PaperB} starts from
the hypothesis $\calB_{\beta}(\Wset_N;\Xsurf)\leq K$ and derives quantization,
discrepancy, and the variational theory of minimizers, while \cite{PaperC}
develops the intrinsic Shub--Smale condition theory of the divisor map
$\bP H^0(\Xsurf,L^N)\longrightarrow\operatorname{Sym}^N(\Xsurf)$.  The
present paper supplies
\Cref{thm:main}, which makes the bounded-energy hypothesis of \cite{PaperB}
uniformly attainable, and \Cref{thm:curve-transversality}, whose uniform
rescaled $C^3$ estimates and zero-section transversality supply the transverse
families used in \cite{PaperC}; \Cref{thm:intro-geometry} records their
additional geometry.  The uniform small-exponent upper bound in \cite{PaperB}
also uses the logarithmic second-moment estimate of \Cref{thm:log-L1}.

\subsection{Organization}\label{subsec:organization}

\Cref{sec:green-sections} establishes the Green--Poincar\'e--Lelong identity.
\Cref{sec:transversality} states the quantitative transversality input,
and \Cref{sec:direct-transversality-appendix} gives its direct construction
using peak sections, a polynomial tube estimate, and finite-color induction.
\Cref{sec:local-lemma,sec:zero-geometry,sec:log-integrability} derive the
local zero model, separation, and uniform logarithmic integrability.
\Cref{sec:main-proof} completes the holomorphic proof (Route A).
\Cref{sec:route-b} supplies the independent spectral-design proof (Route B).
\Cref{sec:consequences} derives \Cref{cor:intro-equidistribution} and compares
the constructions. Finally, \Cref{app:rescaling} records the rescaled-coordinate
conventions used in the local arguments.

\section{Green kernels and divisor sections}\label{sec:green-sections}

\subsection{Joint regularity at the diagonal}

The expansion \eqref{eq:green-local} has a standard parameter-dependent form:
in a holomorphic coordinate $\xi$ and on a relatively compact subdisk,
\begin{equation}\label{eq:green-joint-local}
 U_{\Xsurf}\bigl(\xi^{-1}(x),\xi^{-1}(q)\bigr)=\log|x-q|+R_{\xi}(x,q),
\end{equation}
with $R_{\xi}$ smooth in \emph{both} variables and with all finite-order norms
bounded uniformly over a fixed finite atlas.  Two consequences follow:
$\e^{\gamma U_{\Xsurf}}$ extends continuously by zero across the diagonal for every
$\gamma>0$, and symmetry gives $\ddc_wU_{\Xsurf}(z,w)=\delta_z-\pmeas$ in the pole
variable.  The proof is
a removable-singularity argument followed by elliptic regularity with
parameters, and we do not repeat it: see
\cite[Lemma~2.1.1]{PaperB}, where joint regularity is needed because the poles are
moved in variational arguments.  In the present paper the poles are fixed, and
only the one-variable expansion \eqref{eq:green-local} is needed for the main
packing construction.  Joint continuity is also used in
\Cref{sec:sharpness} for fixed-degree attainment.

\subsection{A degree-one positive line bundle}

Every compact Riemann surface has holomorphic line bundles of degree one;
for example, $\cO(p)$ has degree one for any $p\in \Xsurf$. We need a metric
whose curvature is the prescribed form $\pmeas$.

\begin{lemma}\label{lem:metric-curvature}
There is a degree-one holomorphic line bundle $L\to \Xsurf$ and a smooth
Hermitian metric $\heo$ on $L$ such that
\[
 c_1(L,\heo)=\pmeas.
\]
\end{lemma}

\begin{proof}
Choose any degree-one holomorphic line bundle $L$ and any smooth Hermitian
metric $\heo_0$.  Both $c_1(L,\heo_0)$ and $\pmeas$ have integral one.  Hence
their
difference has integral zero.  On a compact Riemann surface, solving the
scalar Poisson equation (equivalently, applying the
$\partial\bar\partial$-lemma) gives a smooth real function $\varphi$ with
\[
 \pmeas-c_1(L,\heo_0)=\ddc\varphi.
\]
With the convention $|e|_{\heo}=\e^{-\varphi}|e|_{\heo_0}$ in a local holomorphic
frame, $-\ddc\log|e|_{\heo}$ differs from
$-\ddc\log|e|_{\heo_0}$ by $\ddc\varphi$.  Thus the resulting metric $\heo$ has
curvature $\pmeas$.
\end{proof}

Fix such $(L,\heo)$.  The tensor power $(L^N,\heo^N)$ has curvature $N\pmeas$
and degree $N$.

\subsection{Poincar\'e-Lelong and the exact potential identity}

Let $s\in H^0(\Xsurf,L^N)$ be nonzero.  Write its zero divisor, including
multiplicities, as
\[
 Z_s=\sum_{j=1}^{N}w_j.
\]
The total multiplicity is $N$ because
$\deg(L^N)=\int_{\Xsurf} c_1(L^N,\heo^N)=N\int_{\Xsurf}\pmeas=N$.
The Poincar\'e-Lelong formula, in the convention of
Remark \ref{rem:conventions} and as in \cite[Ch.~III]{DemaillyCADG} or
\cite[Ch.~1]{GriffithsHarris}, is
\begin{equation}\label{eq:PL}
 \ddc\log|s|_{\heo^N}=[Z_s]-N\pmeas.
\end{equation}
Indeed, in a local holomorphic frame $e$ write $s=f e^N$.  Then
$\log|s|_{\heo^N}=\log|f|+N\log|e|_{\heo}$.  The scalar
Poincar\'e-Lelong formula gives $\ddc\log|f|=[Z_f]$, while the curvature
convention gives $-\ddc\log|e|_{\heo}=c_1(L,\heo)=\pmeas$.  Adding the two
identities proves \eqref{eq:PL} locally, and both sides patch globally.
On the other hand, by \eqref{eq:green-normalization},
\begin{equation}\label{eq:green-divisor-ddc}
 \ddc\sum_{j=1}^{N}U_{\Xsurf}(\,\cdot\,,w_j)=[Z_s]-N\pmeas.
\end{equation}
Here $\cdot$ denotes the free first variable of the Green kernel, and the
identity is an equality of distributions (equivalently, of signed
measures after applying $\ddc$).

The equality of the two curvature currents implies that the logarithmic section norm and the divisor Green potential differ by a constant.  Mean-zero normalization determines that constant.

\begin{proposition}\label{prop:potential-identity}
For every nonzero $s\in H^0(\Xsurf,L^N)$,
\begin{equation}\label{eq:potential-identity}
 \sum_{j=1}^{N}U_{\Xsurf}(z,w_j)
 =\log|s(z)|_{\heo^N}
 -\int_{\Xsurf}\log|s|_{\heo^N}\diff\pmeas.
\end{equation}
\end{proposition}

\begin{proof}
By \eqref{eq:PL} and \eqref{eq:green-divisor-ddc}, the difference of the
two sides before subtracting the integral is a distributional harmonic
function.  Weyl's lemma makes it smooth, and the maximum principle on the
connected compact surface makes it constant.  Both sides are locally
integrable because $\log|z|\in L^1_{\rm loc}(\C)$.  The Green sum has
integral zero by \eqref{eq:green-normalization}; integrating the constant
identity against the probability form $\pmeas$ therefore identifies the
constant as $-\int_{\Xsurf}\log|s|_{\heo^N}\,\diff\pmeas$ and proves
\eqref{eq:potential-identity}.
\end{proof}

Multiplying $s$ by a nonzero scalar changes both logarithmic terms on the
right of \eqref{eq:potential-identity} by the same constant, so the identity
depends only on the divisor.

\section{Quantitative zero-section transversality}
\label{sec:transversality}

\subsection{Rescaled norms and quantitative transversality}

Choose a Riemannian metric $\geo$ compatible with the complex structure and write
$\geo_N=N\geo$.  With $\nabla$ denoting the Chern connection of $(L^N,\heo^N)$, write
\begin{equation}\label{eq:Ck-rescaled}
 \|s\|_{C^r_N}:=
 \max_{0\leq j\leq r}\sup_{x\in \Xsurf}
 |\nabla^js(x)|_{\geo_N,\heo^N}.
\end{equation}
A section is called uniformly $\eta$-transverse to the zero section when
\begin{equation}\label{eq:eta-transverse}
 |s(x)|_{\heo^N}+|\nabla s(x)|_{\geo_N,\heo^N}\geq\eta
 \qquad(x\in \Xsurf).
\end{equation}
Within this paper we abbreviate this condition to
\emph{$\eta$-zero-section transversality}.  It means that the section meets
the zero section with a derivative bounded away from zero wherever its
value is small.  This terminology distinguishes the scalar condition used
here from the more general relative transversality problems in the
Donaldson--Auroux framework.

We use the integrable form of the Donaldson--Auroux quantitative
transversality theorem, specialized to uniform transversality to the zero
section.  Mohsen's formulation is particularly convenient
because it records both compactness in rescaled charts and transversality of
all limit sections.

The holomorphic construction rests on the following uniform quantitative transversality statement.

\begin{theorem}
\label{thm:curve-transversality}
There are constants $C_r<\infty$, $0\leq r\leq3$, and $\eta>0$ such that,
for every sufficiently large $N$, some holomorphic section
$s_N\in H^0(\Xsurf,L^N)$ satisfies
\begin{equation}\label{eq:our-uniform-estimates}
 \|s_N\|_{C^r_N}\leq C_r\quad(0\leq r\leq3),
 \qquad
 |s_N|_{\heo^N}+|\nabla s_N|_{\geo_N,\heo^N}\geq\eta\quad\text{on }\Xsurf.
\end{equation}
\end{theorem}

This theorem is the only external construction used in Route A, and it is
quoted rather than reproved.  The direct construction in Section~\ref{sec:direct-transversality-appendix} gives an independent,
curve-specific construction of sections with the same properties.  The
displayed estimates provide the interface between the two constructions, and
\Cref{prop:transversality-to-packing} shows that no other property of the
sections is needed.

\begin{proof}
The cited theorem supplies the sections and the rescaled derivative bounds;
what requires an argument is the passage from transversality of each
\emph{limit} section to a single constant $\eta$ valid for all $N$ at once,
and that is where compactness of $\Xsurf$ enters.

Apply \cite[Theorem~2.2 and its proof]{MohsenLimit} to the positive
Hermitian line bundle $(L,\heo)$.  In the integrable case the resulting
sections are holomorphic.  The estimates displayed in Mohsen's proof give,
for every fixed $r$, a bound
$\|s_N\|_{C^r_N}=O_r(1)$ in the renormalized K\"ahler metric determined by
$\pmeas$.  Any two fixed metrics on the compact surface are uniformly
equivalent. For higher covariant derivatives, the difference of their
Levi--Civita connections is a fixed smooth tensor, so changing connections
adds lower-order derivatives with uniformly bounded rescaled coefficients.
The same estimates therefore hold, after changing constants, for $\geo_N=N\geo$.

For completeness, the uniform lower bound follows from the limit-section
formulation.  If no $\eta>0$ worked, one could choose points $x_N$ at which
the value-plus-first-jet norm tends to zero.  Recenter the rescaled charts
at $x_N$.  Renormalizability gives a subsequence converging in $C^1$ on
compact sets to a limit section whose value and derivative both vanish at
the origin.  This contradicts the transversality of every limit section in
Mohsen's theorem.  Hence there is a constant $\eta>0$, independent of $N$,
for which
\[
 \inf_{x\in \Xsurf}\bigl(|s_N(x)|_{\heo^N}
       +|\nabla s_N(x)|_{\geo_N,\heo^N}\bigr)\geq\eta.
\]
Taking $0\leq r\leq3$ gives \eqref{eq:our-uniform-estimates}.
\end{proof}

\subsection{Source conventions and uniform estimates}

The rescaled compactness statement is used with the Chern connection and Hermitian norms of the manuscript.  The following remarks specify this translation.

\begin{remark}\label{rem:mohsen-translation}
In the published version of \cite{MohsenLimit}, the integrable holomorphic
statement is Theorem~2.2, and Definition~1.1 specifies renormalizability.
The derivative estimates used above occur in the joint proof of Theorems~2.2
and~2.3, which is why the citation includes the proof.  Three features of the
source are used.  First, in the integrable case the sections produced are
holomorphic.  Second, the theorem is formulated through renormalized charts,
and renormalizability gives compactness of all recentered subsequences in the
$C^1$ topology on compact sets, which is the derivative order used in the
contradiction argument above; the higher orders $r\leq3$ enter only through the
displayed $C^r_N$ bounds.  Third, every limit section obtained from such a
recentering is transverse, which is what forbids a simultaneous vanishing of
value and first derivative at the origin.  The passage from Mohsen's
renormalized K\"ahler metric to the fixed choice $\geo_N=N\geo$ changes only
constants: the underlying metrics are fixed and, $\Xsurf$ being compact,
uniformly equivalent.  Compactness is also what converts the pointwise
transversality of each limit section into the single constant $\eta>0$,
uniform in $N$ and in the base point, asserted in
\eqref{eq:our-uniform-estimates}.
\end{remark}

On a Riemann surface, transversality to the zero section amounts to nonvanishing of the first jet at each zero.

\begin{remark}
In the general Donaldson--Auroux framework, quantitative transversality may
be imposed relative to a prescribed submanifold or stratification.  The
specialization used here is the elementary curve case in which the target
is the zero section of $L^N$.  Thus the only quantitative condition carried
forward from Theorem \ref{thm:curve-transversality} is exactly
\eqref{eq:eta-transverse}; it forces all zeros to be simple and uniformly
nondegenerate in the $\geo_N$ scale.
\end{remark}

An alternative construction specialized to Riemann surfaces is available in the companion work.

\begin{remark}\label{rem:curve-specific}
\Cref{thm:curve-transversality} is the only transversality result used in the
main argument, and the only place where an external construction of sections
is used.  Readers who prefer not to invoke the general theorem may instead
use the curve-specific derivation given in Section~\ref{sec:direct-transversality-appendix}.  That derivation is
independent of the general quantitative-transversality theorem and of
Bergman-kernel asymptotics: it builds its peak sections from a cutoff
Bargmann-Fock model corrected by a weighted $L^2$ solution of the
$\bar\partial$ equation, and it proves the tube
estimate it needs from complex Crofton area and the Lelong--Jensen density
bound.
It uses two complex local parameters, polynomial approximation, and a
finite-color correction.  Nothing in the present paper depends on its internal
details: by Proposition \ref{prop:transversality-to-packing}, any
construction supplying the two displayed hypotheses of
\eqref{eq:our-uniform-estimates} yields the Green-packing conclusion unchanged.
\end{remark}

\section{A direct two-parameter transversality construction}
\label{sec:direct-transversality-appendix}
This part supplies a curve-specific proof of the uniform transversality theorem used in Route A.
The general Donaldson--Auroux--Mohsen proof in Section~\ref{sec:transversality} remains an independent option.
The direct method yields an explicit quantitative estimated-Sard bound and a constructive finite-color induction.
\subsection{Direct construction: introduction}\label{direct:sec:intro}

\subsubsection{The statement}\label{direct:subsec:statement}

Let $\Xsurf$ be a connected compact Riemann surface, $\pmeas$ a smooth positive
$(1,1)$-form with $\int_{\Xsurf} \diff\pmeas=1$, and $(L,\heo)$ a holomorphic line bundle of
degree one with a Hermitian metric of curvature $c_1(L,\heo)=\pmeas$.  Fix a
compatible Riemannian metric $\geo$, and write $\geo_N=N\geo$.  For the Chern connection
$\nabla$ of $(L^N,\heo^N)$, the rescaled $C^r$ norm is
\begin{equation}\label{direct:eq:Ck-rescaled}
 \|s\|_{C^r_N}:=\max_{0\leq j\leq r}\sup_{x\in \Xsurf}
 |\nabla^js(x)|_{\geo_N,\heo^N},
\end{equation}
The section $s$ is \emph{$\eta$-zero-section transverse} when
\begin{equation}\label{direct:eq:eta-transverse}
 |s(x)|_{\heo^N}+|\nabla s(x)|_{\geo_N,\heo^N}\geq\eta
 \qquad(x\in \Xsurf).
\end{equation}
The main result is the following existence theorem.

\begin{theorem}[Quantitative transversality on curves]
\label{direct:thm:direct-curve-transversality}
There are constants $C_r<\infty$ for $0\leq r\leq3$ and $\eta>0$, all
independent of $N$, such that for every sufficiently large $N$ some
$s_N\in H^0(\Xsurf,L^N)$ satisfies
\begin{equation}\label{direct:eq:appendix-uniform-estimates}
 \|s_N\|_{C^r_N}\leq C_r\quad(0\leq r\leq3),
 \qquad
 |s_N|_{\heo^N}+|\nabla s_N|_{\geo_N,\heo^N}\geq\eta\quad\text{on }\Xsurf.
\end{equation}
\end{theorem}

The conclusion is not new.  It specializes the integrable form of the
Donaldson--Auroux quantitative-transversality
theorem~\cite{Donaldson,Auroux,MohsenAFST} to a curve and to the zero section as
target.  The integrable holomorphic statement appears in
\cite[Theorem~2.2 and its proof]{MohsenLimit}, and the first construction here follows that
route.  The affine perturbation and the associated first-jet map already
appear in Auroux's method~\cite{AurouxRemark}. Here we give a direct
holomorphic implementation on curves, with a linear-degree tube estimate
and the explicit logarithmic exponent $1/2$. We restate the theorem to
make the argument self-contained and to permit comparison with the general
construction.

Sections satisfying these two estimates are precisely those required by the
Green-packing construction in \Cref{prop:transversality-to-packing} and by the condition theory
of~\cite{PaperC}. The construction here may instead use the general theorem; the
packing deduction does not depend on the choice of derivation.

\subsubsection{The mechanism and its scope}\label{direct:subsec:novelty}

The paper develops a curve-specific refinement of Auroux's affine-jet
method~\cite{AurouxRemark}. The local avoidance estimate and the constants
in the global induction are the points requiring quantitative control.

Donaldson's scheme perturbs a section successively over the cells of a net,
using localized peak sections, and at each step one must gain transversality on
the new cell without destroying the gains on earlier cells. This requires a
quantitative local avoidance statement---an ``estimated Sard lemma''---and the
arithmetic of the induction is governed by how much
transversality one gains per step relative to the perturbation size.

We perturb in \emph{two} complex parameters,
\begin{equation}\label{direct:eq:two-parameter-perturbation}
 f(u)\longmapsto f(u)+w_0+w_1u,
\end{equation}
which is natural in complex dimension one because value and first derivative
at a point are two independent scalar conditions on a holomorphic function of
one variable. The exactly nontransverse parameters form the image in $\C^2$ of the
holomorphic map $F_f$ in~\eqref{direct:eq:bad-parameter-map}; this image may
degenerate to a point. A thin tube around a
two-real-dimensional set cannot cover a four-dimensional parameter ball, so
the resulting loss is only logarithmic:
\begin{equation}\label{direct:eq:intro-Sard-scale}
 |f(u)+w_0+w_1u|+|f'(u)+w_1|
 \geq c\,\delta\,|\log\delta|^{-1/2}
 \qquad(|u|\leq1).
\end{equation}
This is the content of~\Cref{direct:prop:estimated-sard}. A logarithmic loss is
precisely what the
finite-color induction can absorb: with $Q_0=O(D^2)$ colors the accumulated
logarithm is $O(D^2\log D)$, so the admissible same-color loss is bounded below
by an inverse power of $D^2\log(D+2)$, and this dominates the exponentially
small tail $e^{-cD}$ once $D$ is fixed sufficiently large. See
\Cref{direct:sec:induction} and~\Cref{direct:rem:why-two-parameters} for the detailed
comparison. Classical one-parameter estimated-Sard methods provide another
route to quantitative transversality; no impossibility claim for those methods
is needed or proved here. The advantage of the present parametrization is that
one distance estimate in $\C^2$ controls both components of the first jet.

\subsubsection{The two analytic ingredients}
\label{direct:subsec:analytic-inputs}

Two ingredients of the scheme are genuinely analytic, and both are established
here rather than quoted.

\emph{Peak sections.} The construction needs, at every center $p$ and for
every large $N$, two sections with prescribed value-and-derivative jets at $p$
and with exponentially decaying rescaled derivatives away from $p$. One route
is to read them off the near-diagonal and off-diagonal asymptotics of the
Bergman kernel. We instead give a direct construction
(\Cref{direct:sec:peaks}): cut off the Bargmann--Fock models $1$ and $u$ at radius
$R_N=A_{\rm pk}\sqrt{\log N}$ and correct the resulting $\bar\partial$ defect by the
weighted $L^2$ estimate of \Cref{direct:lem:appendix-weighted-dbar}, whose weight
$e^{2a\rho_{p,N}}$ is built from a smoothed rescaled distance. Because the
defect is supported in an annulus where the Gaussian weight is
$e^{-cR_N^2}=N^{-cA_{\rm pk}^2}$, for each prescribed $K$ a sufficiently
large fixed $A_{\rm pk}$ makes the correction small enough that rescaled
Cauchy estimates give the local models with error $O(N^{-K})$. Inverting a
$2\times2$ jet matrix of the form $\Id+O(N^{-K})$ makes the jets exact.  No
Bergman asymptotics are used.

\emph{The tube estimate.} For the estimated Sard lemma, we need an upper bound on
the volume of an $r$-tube around a polynomially parametrized complex curve in a
ball of radius $R$. This can be quoted from the general real-algebraic tube
theorem of Wongkew~\cite{Wongkew}, at the cost of a high power of the degree.
We prove instead the sharper
\begin{equation}\label{direct:eq:intro-tube}
 \vol_4\bigl(N_r(P(\C))\cap B(0,R)\bigr)\leq C\,d\,R^2r^2,
 \qquad 0<r<R,
\end{equation}
with \emph{linear} dependence on $\deg P\leq d$
(\Cref{direct:lem:polynomial-tube}). The proof is specific to complex curves. Complex
Crofton bounds the area of an algebraic curve linearly in its degree
(\Cref{direct:lem:complex-crofton-area}), while the Lelong--Jensen monotonicity formula
bounds the area of the curve in a small ball from \emph{below}, at both singular
and smooth points (\Cref{direct:lem:complex-lower-density}).  Together these bound
the cardinality of a maximal $r$-separated net on the curve, and the tube is then
covered by balls around that net. The linear degree bound gives the exponent
$1/2$ in \eqref{direct:eq:intro-Sard-scale}; a coarser polynomial dependence would
give a larger but still finite logarithmic loss. The induction in
\Cref{direct:sec:induction} closes for every fixed finite nonnegative logarithmic exponent, so the
sharper tube estimate improves the quantitative local statement rather than
serving as a hidden existence assumption.

Consequently, the only external analytic existence result used for
\Cref{direct:thm:direct-curve-transversality} is the Bochner--Kodaira--H\"ormander
$L^2$ theorem for $\bar\partial$ on a positive line bundle. We also use standard
bounded-geometry smoothing, B\'ezout's theorem, and Lelong--Jensen monotonicity.
We use no quantitative-transversality theorem and no Bergman-kernel expansion.

\subsubsection{Organization}\label{direct:subsec:organization}

\Cref{direct:sec:peaks} constructs the peak sections. \Cref{direct:sec:tube} proves the
Crofton--Lelong tube estimate and the two-parameter estimated Sard lemma.
\Cref{direct:sec:coloring} constructs the colored net and analyzes the correction for
one color. \Cref{direct:sec:induction} fixes the order in which the constants are
chosen, states the induction invariant, and proves
\Cref{direct:thm:direct-curve-transversality}. \Cref{direct:app:rescaling} records the
rescaled coordinates and frames used throughout.

\subsection{Rescaled norms and peak sections}\label{direct:sec:peaks}

Throughout we use the rescaled norms \eqref{direct:eq:Ck-rescaled} and the
transversality condition \eqref{direct:eq:eta-transverse}, and we work in the
controlled rescaled holomorphic coordinates and frames of
\Cref{direct:app:rescaling}. This section constructs the two peak sections used as
perturbation directions in the correction scheme.

We derive the peak sections directly from a cutoff Bargmann--Fock model and a
weighted $L^2$ solution of the $\bar\partial$ equation. We use the
standard smooth approximation of distance functions and the basic
H\"ormander estimate; see~\cite{GreeneWu,Hormander,DemaillyCADG}. No
Bergman-kernel expansion is needed in this section. For background on Bergman-kernel expansions, see~\cite{MaMarinescu}.
The uniform exponential off-diagonal estimate is given in
\cite[Theorems~1 and~5]{MaMarinescuExponential}. These results provide an
alternative analytic framework; they are not used in the construction below.

Fix $\epsilon_*=1/2$ throughout, and increase the lower threshold for $N$ if
necessary. The uniform metric comparison of~\Cref{direct:app:rescaling} then allows
us to fix
constants $r_{\rm cov}\in(0,1)$ and $R_{\rm E}>2$ with
\begin{equation}\label{direct:eq:peak-scale-containments}
 B_{\geo_N}(p,2r_{\rm cov})\subset\{|u|<1\}
 \subset\{|u|<1+\epsilon_*\}
 \subset\{|u|<R_{\rm E}\}.
\end{equation}
Indeed, on a fixed larger coordinate disk the metric is bounded above and below
by fixed multiples of the Euclidean metric. Thus the first inclusion follows
by taking $r_{\rm cov}$ small; then choose $R_{\rm E}>1+\epsilon_*$. Both
constants are chosen before the net and the color separation below.

\begin{lemma}[Smoothed rescaled distance]
\label{direct:lem:appendix-smoothed-distance}
There are constants $C_0,C_1,C_2$ such that, for every sufficiently large
$N$ and every $p\in \Xsurf$, there is a smooth function
$\rho_{p,N}:\Xsurf\to[0,\infty)$ satisfying
\begin{equation}\label{direct:eq:appendix-smoothed-distance}
 |\rho_{p,N}(x)-d_{\geo_N}(p,x)|\leq C_0,
 \qquad
 |\nabla\rho_{p,N}|_{\geo_N}\leq C_1,
 \qquad
 |\nabla^2\rho_{p,N}|_{\geo_N}\leq C_2.
\end{equation}
\end{lemma}

\begin{proof}
The metrics $\geo_N=N\geo$ have bounded geometry uniformly in $N$: their injectivity
radii are bounded below (indeed, they grow like $\sqrt N$), while their
curvatures and curvature derivatives are uniformly bounded. Choose a
uniformly locally finite atlas of fixed-radius normal balls and a subordinate
partition of unity with uniformly bounded first two derivatives. In each
chart, convolve the $1$-Lipschitz function $d_{\geo_N}(p,\cdot)$ with a fixed
small mollifier. The resulting local functions differ from the distance by a
uniform constant and have uniformly bounded first and second derivatives.
On overlaps, any two such local regularizations differ by a uniform constant,
because both approximate the same distance function. Hence the terms produced
by differentiating the partition of unity are uniformly bounded. Patching gives
a smooth function with the three stated estimates. Adding one uniform constant
makes it nonnegative. This is the standard bounded-geometry smoothing argument
for Lipschitz functions; compare~\cite{GreeneWu}.
\end{proof}

\begin{lemma}[Uniform weighted $\bar\partial$ estimate]
\label{direct:lem:appendix-weighted-dbar}
There are $a_0>0$ and $C<\infty$ such that, for every $0<a\leq a_0$, every
large $N$, every $p\in \Xsurf$, and every smooth $L^N$-valued $(0,1)$-form
$\alpha$, one can solve $\bar\partial q=\alpha$ with
\begin{equation}\label{direct:eq:appendix-weighted-dbar}
 \int_{\Xsurf}|q|_{\heo^N}^2e^{2a\rho_{p,N}}\,\diff\pmeas
 \leq C\int_{\Xsurf}|\alpha|_{\geo_N,\heo^N}^2
              e^{2a\rho_{p,N}}\,\diff\pmeas.
\end{equation}
\end{lemma}

\begin{proof}
The geometric bounds in \Cref{direct:lem:appendix-smoothed-distance}
are uniform in the center $p$ and in $N$. Hence the rescaled Hessian
bounds and the curvature lower bound used below have constants
independent of both parameters.

Equip $L^N$ with the modified metric $\heo^Ne^{2a\rho_{p,N}}$. Its curvature is
\[
 Ni\Theta(L,\heo)-2a i\partial\bar\partial\rho_{p,N}.
\]
Relative to $\geo_N$, the first term has a fixed positive lower bound, while
\Cref{direct:lem:appendix-smoothed-distance} uniformly bounds the second. Choose
$a_0$ small enough to retain half of that lower bound. To apply the
Bochner--Kodaira--H\"ormander estimate to sections, use
$E_N=L^N\otimes K_{\Xsurf}^{-1}$, with the metric on $K_{\Xsurf}^{-1}$ induced by $\geo_N$.
The canonical identification of $L^N$-valued $(0,q)$-forms with
$E_N$-valued $(1,q)$-forms preserves pointwise norms. The additional
curvature $i\Theta(K_{\Xsurf}^{-1})$ has $\geo_N$-norm $O(N^{-1})$, uniformly on $\Xsurf$.
Consequently, after increasing the threshold for $N$, the curvature
operator for $E_N$ in bidegree $(1,1)$ has a uniform positive lower bound.
The complete K\"ahler estimate~\cite[Chapter~VIII, Theorem~4.5]{DemaillyCADG}
therefore gives
\[
 \int_{\Xsurf} |q|_{\heo^N}^2e^{2a\rho_{p,N}}\,dV_{\geo_N}
 \leq C\int_{\Xsurf} |\alpha|_{\geo_N,\heo^N}^2e^{2a\rho_{p,N}}\,dV_{\geo_N}.
\]
Here $dV_{\geo_N}=N b\,\diff\pmeas$, where $b=dV_{\geo}/\diff\pmeas$ is bounded above and below
by positive constants. Cancelling $N$ and comparing $b$ gives
\eqref{direct:eq:appendix-weighted-dbar}. Every $(0,1)$-form on a curve is
$\bar\partial$-closed, and elliptic regularity gives a smooth solution
for smooth $\alpha$.
\end{proof}

Choose curvature-normal holomorphic coordinates and frames, uniformly in the
center, so that $u=\sqrt N z$ and
\begin{equation}\label{direct:eq:appendix-BF-weight}
 |e_p(z)|_{\heo}^2=e^{-\phi_p(z)},
 \qquad
 \phi_p(z)=|z|^2+O(|z|^3),
 \qquad
 \diff\pmeas=N^{-1}m_p(u/\sqrt N)dA(u),
\end{equation}
with uniform derivative bounds. The cutoff radius is
\begin{equation}\label{direct:eq:appendix-cutoff-radius}
 R_N=A_{\rm pk}\sqrt{\log N},
\end{equation}
where the fixed constant $A_{\rm pk}$ will be chosen large. Since
$R_N^3/\sqrt N=A_{\rm pk}^3(\log N)^{3/2}/\sqrt N\to0$,
the cubic Taylor remainder in the rescaled weight remains uniformly small
on the growing cutoff disks. In particular, on $|u|\leq2R_N$,
\begin{equation}\label{direct:eq:appendix-Gaussian-bound}
 c|u|^2\leq N\phi_p(u/\sqrt N)\leq C|u|^2
\end{equation}
for all sufficiently large $N$.

Let $\chi\in C_c^\infty([0,\infty))$ equal one on $[0,1]$ and zero on
$[4,\infty)$. Extend by zero the smooth local sections
\begin{equation}\label{direct:eq:appendix-BF-models}
 m^0_{p,N}:=\chi\!\left(\frac{|u|^2}{R_N^2}\right)e_p^N,
 \qquad
 m^1_{p,N}:=\chi\!\left(\frac{|u|^2}{R_N^2}\right)u\,e_p^N.
\end{equation}
They are holomorphic on $|u|<R_N$. From
\eqref{direct:eq:appendix-smoothed-distance} and
\eqref{direct:eq:appendix-Gaussian-bound}, for $j=0,1$,
\begin{align}
 \int_{\Xsurf}|m^j_{p,N}|_{\heo^N}^2e^{2a\rho_{p,N}}\,\diff\pmeas
 &\leq \frac{C}{N},
 \label{direct:eq:appendix-model-L2}\\
 \int_{\Xsurf}|\bar\partial m^j_{p,N}|_{\geo_N,\heo^N}^2
           e^{2a\rho_{p,N}}\,\diff\pmeas
 &\leq \frac{C}{N}R_N^{2j}e^{-cR_N^2+CR_N}.
 \label{direct:eq:appendix-defect-L2}
\end{align}
Indeed, the defect is supported in $R_N\leq|u|\leq2R_N$. The rescaled
norm of the cutoff derivative is $O(R_N^{-1})$, the annulus has
$\diff\pmeas$-area $O(N^{-1}R_N^2)$, and the Gaussian factor is
$e^{-cR_N^2}$ there.

Solve
\[
 \bar\partial q^j_{p,N}=\bar\partial m^j_{p,N}
\]
by~\Cref{direct:lem:appendix-weighted-dbar}. The globally holomorphic raw sections are
\begin{equation}\label{direct:eq:appendix-raw-peaks}
 \widetilde\sigma^j_{p,N}:=m^j_{p,N}-q^j_{p,N}.
\end{equation}
Given any fixed positive integer $K$, a
sufficiently large choice of $A_{\rm pk}$ in
\eqref{direct:eq:appendix-cutoff-radius} makes
\begin{equation}\label{direct:eq:appendix-correction-small}
 \int_{\Xsurf}|q^j_{p,N}|_{\heo^N}^2e^{2a\rho_{p,N}}\,\diff\pmeas
 \leq N^{-2K-2}.
\end{equation}
On a fixed disk $|u|\leq R_{\rm E}+1$, the correction is holomorphic.
Rescaled Cauchy estimates and
\eqref{direct:eq:appendix-correction-small} therefore give, in the frame $e_p^N$,
\begin{equation}\label{direct:eq:appendix-raw-local-model}
 \|\widetilde f^0_{p,N}-1\|_{C^3}
 +\|\widetilde f^1_{p,N}-u\|_{C^3}
 \leq C_KN^{-K}.
\end{equation}
Moreover, \eqref{direct:eq:appendix-model-L2} and
\eqref{direct:eq:appendix-correction-small} imply
\begin{equation}\label{direct:eq:appendix-raw-weighted-bound}
 \int_{\Xsurf}|\widetilde\sigma^j_{p,N}|_{\heo^N}^2
      e^{2a\rho_{p,N}}\,\diff\pmeas\leq\frac{C}{N}.
\end{equation}

\begin{proposition}[Value-and-derivative peak sections]
\label{direct:prop:peaks}
There are constants $c>0$ and $C_r<\infty$ such that, for all
large $N$ and every $p\in \Xsurf$, there are sections
$\sigma^0_{p,N},\sigma^1_{p,N}\in H^0(\Xsurf,L^N)$ satisfying:
\begin{enumerate}[label=\textup{(\roman*)}]
\item $|\sigma^0_{p,N}(p)|_{\heo^N}=1$ and
$\nabla\sigma^0_{p,N}(p)=0$;
\item in a rescaled coordinate $u$ centered at $p$ and valid for
$|u|<R_{\rm E}$, the section $\sigma^0_{p,N}$ has no zero and
\begin{equation}\label{direct:eq:peak-ratio-model}
 \frac{\sigma^1_{p,N}}{\sigma^0_{p,N}}(u)
 =u+\varepsilon_{p,N}(u),
 \qquad
 \|\varepsilon_{p,N}\|_{C^2(|u|<R_{\rm E})}
 \leq C N^{-1/2},
\end{equation}
and the representing function of $\sigma^0_{p,N}$ and its reciprocal have
uniformly bounded $C^3$ norms;
\item for $j=0,1$, $0\leq r\leq3$, and $x\in \Xsurf$,
\begin{equation}\label{direct:eq:peak-decay}
 |\nabla^r\sigma^j_{p,N}(x)|_{\geo_N,\heo^N}
 \leq C_r e^{-c d_{\geo_N}(p,x)}.
\end{equation}
\end{enumerate}
All constants are independent of $p$ and $N$.
\end{proposition}

\begin{proof}
The raw sections $\widetilde\sigma^j_{p,N}$ already have the required local
models up to $O(N^{-K})$; what they lack is exact jets. We correct them by
inverting a
$2\times2$ matrix which is $\Id+O(N^{-K})$, which changes neither the local
estimates nor the global tail, and then read off the three assertions.

Let $G_{p,N}$ be the value-and-$u$-derivative jet matrix of the two raw
sections at the origin:
\[
 G_{p,N}:=
 \begin{pmatrix}
  \widetilde f^0_{p,N}(0)&\widetilde f^1_{p,N}(0)\\
  (\widetilde f^0_{p,N})'(0)&(\widetilde f^1_{p,N})'(0)
 \end{pmatrix}.
\]
By~\eqref{direct:eq:appendix-raw-local-model},
$G_{p,N}=\Id+O(N^{-K})$, so it is uniformly invertible. The exact-jet sections
are
\[
 (\sigma^0_{p,N},\sigma^1_{p,N})
 :=(\widetilde\sigma^0_{p,N},\widetilde\sigma^1_{p,N})G_{p,N}^{-1}.
\]
The two jets are now exactly $(1,0)$ and $(0,1)$. Since the chosen frame
has zero Chern connection form at the center, this gives part~\textup{(i)}.
It also preserves
\eqref{direct:eq:appendix-raw-local-model}; hence, on the fixed disk,
\[
 f^0_{p,N}=1+O_{C^3}(N^{-K}),
 \qquad
 f^1_{p,N}=u+O_{C^3}(N^{-K}).
\]
For $N$ large, the first function is zero-free, its reciprocal has uniform
$C^3$ bounds, and the quotient is
$u+O_{C^2}(N^{-K})$.  Weakening this with $K\geq1$ gives
\eqref{direct:eq:peak-ratio-model}.

It remains to prove the global tail. The transformation by
$G_{p,N}^{-1}$ preserves \eqref{direct:eq:appendix-raw-weighted-bound}. On a fixed
$\geo_N$-ball centered at $x$, the function $\rho_{p,N}$ differs from
$\rho_{p,N}(x)$ by $O(1)$.  Uniform rescaled Cauchy estimates consequently
give, for $0\leq r\leq3$,
\[
 |\nabla^r\sigma^j_{p,N}(x)|_{\geo_N,\heo^N}^2
 \leq C_rN\int_{B_{\geo_N}(x,1)}|\sigma^j_{p,N}|_{\heo^N}^2\diff\pmeas
 \leq C_re^{-2a\rho_{p,N}(x)}.
\]
Finally, \eqref{direct:eq:appendix-smoothed-distance} replaces $\rho_{p,N}(x)$ by
$d_{\geo_N}(p,x)$ at the cost of one fixed factor.  This proves
\eqref{direct:eq:peak-decay}.
\end{proof}

\subsection{The tube estimate and the estimated Sard lemma}\label{direct:sec:tube}

\subsubsection{A two-parameter estimated Sard lemma}

For a holomorphic function $f$, Auroux's affine-jet method~\cite{AurouxRemark}
identifies the exactly nontransverse parameters for
$f(u)+w_0+w_1u$ as the image of the holomorphic map
\begin{equation}\label{direct:eq:bad-parameter-map}
 F_f(u):=(-f(u)+u f'(u),-f'(u))\in\C^2.
\end{equation}
Indeed, the equations $f(u)+w_0+w_1u=0$ and $f'(u)+w_1=0$ are equivalent
to $(w_0,w_1)=F_f(u)$. Moreover, for $|u|\leq1$,
\begin{equation}\label{direct:eq:jet-parameter-distance}
 c\,\dist(w,F_f(u))
 \leq |f(u)+w_0+w_1u|+|f'(u)+w_1|
 \leq C\,\dist(w,F_f(u)),
\end{equation}
because $(a,b)\mapsto(a+bu,b)$ and its inverse have uniformly bounded
operator norms.
In matrix form the map is
\[
 T_u=\begin{pmatrix}1&u\\0&1\end{pmatrix},
 \qquad
 T_u^{-1}=\begin{pmatrix}1&-u\\0&1\end{pmatrix}.
\]
For $|u|\leq1$, both norms are bounded by an absolute constant. Applying
$T_u$ to $w-F_f(u)$ gives exactly the value-and-derivative pair in
\eqref{direct:eq:jet-parameter-distance}.

Identify $\C^2$ with $\R^4$, and let $\vol_4$ denote four-dimensional
Lebesgue measure, and write $N_r(S)=\{z:\dist(z,S)<r\}$ for the open
$r$-neighborhood of a set $S$. We first prove the required tube estimate directly from
complex Crofton area and Lelong--Jensen density.

\begin{lemma}[Complex Crofton area bound]\label{direct:lem:complex-crofton-area}
Let $\Gamma\subset\C^2$ be a reduced algebraic curve of degree $D$. Then,
for every $S>0$,
\begin{equation}\label{direct:eq:complex-crofton-area}
 \mathcal H^2\bigl(\Gamma\cap B(0,S)\bigr)\leq2\pi DS^2,
\end{equation}
where $\mathcal H^2$ is Euclidean two-dimensional Hausdorff measure.
\end{lemma}

\begin{proof}
The singular locus has $\mathcal H^2$-measure zero. For a unit vector
$v\in\C^2$, let $\pi_v(z)=\langle z,v\rangle$, and average over
$v\in\mathbb P^1$ with the unitary-invariant probability measure $d\nu$.
If $u_x$ is a unit vector spanning the complex tangent line at a regular
point $x\in\Gamma$, then the real two-Jacobian is
\[
 J_2\bigl(\pi_v|_{T_x\Gamma}\bigr)=|\langle u_x,v\rangle|^2,
 \qquad
 \int_{\mathbb P^1}|\langle u_x,v\rangle|^2d\nu(v)=\frac12.
\]
The area formula and Fubini therefore give
\[
 \frac12\mathcal H^2\bigl(\Gamma\cap B(0,S)\bigr)
 =\int_{\mathbb P^1}\int_{\C}N_v(a;S)\,dA(a)\,\diff\nu(v),
\]
where $N_v(a;S)$ counts the regular points of
$\Gamma\cap B(0,S)\cap\pi_v^{-1}(a)$. For almost every direction and fiber,
B\'ezout gives $N_v(a;S)\leq D$, while
$\pi_v(B(0,S))\subset D(0,S)$. The right-hand side is at most
$D\pi S^2$, proving~\eqref{direct:eq:complex-crofton-area}.
\end{proof}

\begin{lemma}[Complex lower density]\label{direct:lem:complex-lower-density}
Let $A\subset\C^2$ be a closed complex analytic set of pure dimension one. For
all $x\in A$ and $\rho>0$,
\begin{equation}\label{direct:eq:complex-lower-density}
 \mathcal H^2\bigl(A\cap B(x,\rho)\bigr)\geq\pi\rho^2.
\end{equation}
\end{lemma}

\begin{proof}
The Lelong--Jensen formula says that
\[
 \rho\longmapsto
 \frac{\mathcal H^2(A\cap B(x,\rho))}{\pi\rho^2}
\]
is nondecreasing, and its limit at zero is the positive integer
multiplicity of $A$ at $x$. Hence the ratio is at least one. See
\cite[Chapter~III, Sections~6--7]{DemaillyCADG}. Equivalently, this is the
Euclidean monotonicity formula for the calibrated current $[A]$.
\end{proof}

\begin{lemma}[Tube bound for a polynomial curve]\label{direct:lem:polynomial-tube}
There is a universal constant $C>0$ such that, if
$P:\C\to\C^2$ is a polynomial map of degree at most $d\geq1$ and
$0<r<R$, then
\begin{equation}\label{direct:eq:polynomial-tube}
 \vol_4\bigl(N_r(P(\C))\cap B_{\C^2}(0,R)\bigr)
 \leq C d R^2r^2.
\end{equation}
\end{lemma}

\begin{proof}
The assertion is immediate when $P$ is constant, so assume otherwise and let
$\Gamma$ be the Zariski closure of $P(\C)$. Then $\Gamma$ is an irreducible
affine plane curve. If $L=\{aw_1+bw_2=c\}$ is a generic complex affine
line, then
\[
 P^{-1}(L)=\{u\in\C:aP_1(u)+bP_2(u)-c=0\}.
\]
The polynomial on the right has degree at most $d$. The image of the
dominant morphism $P:\C\to\Gamma$ contains a nonempty Zariski-open
subset of $\Gamma$. Since $\Gamma$ is an irreducible curve, its complement
is finite. A generic affine line avoids this finite set, meets $\Gamma$
in $\deg\Gamma$ distinct regular points, and has finite transverse
intersections with the parametrization. Each such point has at least one
preimage under $P$, and distinct points have distinct preimages.
The polynomial $aP_1+bP_2-c$ therefore has at least $\deg\Gamma$
distinct roots. Hence
\begin{equation}\label{direct:eq:parametric-curve-degree}
 \deg\Gamma=\#(L\cap\Gamma)\leq d.
\end{equation}

Choose a maximal $r$-separated set
$\{x_1,\ldots,x_M\}\subset\Gamma\cap\overline{B(0,2R)}$. The balls
$B(x_j,r/2)$ are pairwise disjoint, and
\Cref{direct:lem:complex-lower-density} gives at least $\pi r^2/4$ area of
$\Gamma$ in each. Since $r<R$, all these balls lie in $B(0,5R/2)$.
Using~\Cref{direct:lem:complex-crofton-area} and
$D:=\deg\Gamma\leq d$,
\[
 M\frac{\pi r^2}{4}
 \leq\mathcal H^2\bigl(\Gamma\cap B(0,5R/2)\bigr)
 \leq2\pi D(5R/2)^2,
\]
so $M\leq50dR^2/r^2$.

If $z\in N_r(\Gamma)\cap B(0,R)$, choose $y\in\Gamma$ with
$|z-y|<r$. Then $y\in B(0,2R)$, and maximality gives
$|y-x_j|<r$ for some $j$. Hence $|z-x_j|<2r$ and
\[
 N_r(\Gamma)\cap B(0,R)\subset\bigcup_{j=1}^MB(x_j,2r).
\]
Because a radius-$2r$ ball in $\R^4$ has volume $8\pi^2r^4$,
\[
 \vol_4\bigl(N_r(\Gamma)\cap B(0,R)\bigr)
 \leq400\pi^2dR^2r^2.
\]
Finally $P(\C)\subset\Gamma$, so the same estimate holds for $P(\C)$.
\end{proof}

\begin{remark}[Comparison with the real-algebraic tube theorem]
Wongkew's general theorem~\cite{Wongkew} yields a sufficient bound with a
higher power of the degree. The linear dependence above is special to
complex curves: Crofton controls their area linearly in the complex degree,
and Lelong--Jensen supplies a uniform local area lower bound even at
singular points. Thus Wongkew's theorem is no longer a logical prerequisite
here.
\end{remark}

We now state and prove the local avoidance result that drives the construction.
The tube bound enters exactly once, in the volume comparison that rules out the
covering.

\begin{proposition}[Estimated Sard lemma]\label{direct:prop:estimated-sard}
Fix $\epsilon>0$ and $M\geq1$.  There are constants $c>0$ and
$\delta_0>0$, depending only on $\epsilon$ and $M$, such that the statement
below holds with
\begin{equation}\label{direct:eq:Sard-exponent-half}
 \ell_{\rm S}=\frac12.
\end{equation}
If $f$ is holomorphic on $D(0,1+\epsilon)$, $\|f\|_{C^1}\leq M$, and
$0<\delta<\delta_0$, then some $w=(w_0,w_1)\in\C^2$, $|w|\leq\delta$,
satisfies
\begin{equation}\label{direct:eq:strong-local-transversality}
 |f(u)+w_0+w_1u|+|f'(u)+w_1|
 \geq c\delta|\log\delta|^{-\ell_{\rm S}},
 \qquad |u|\leq1.
\end{equation}
\end{proposition}

\begin{proof}
It is enough to prove the result for $M=1$. Indeed, apply that case to
$\widetilde f=f/M$ with perturbation radius $\delta/M$, and then multiply
the chosen parameter by $M$: if $v$ is chosen for $\widetilde f$, take
$w=Mv$. Then
$|w|\leq\delta$, and the normalized lower bound becomes
\[
 M c\frac{\delta}{M}\left|\log\frac{\delta}{M}\right|^{-\ell_{\rm S}}
 \geq c_M\delta|\log\delta|^{-\ell_{\rm S}}
\]
after decreasing $\delta_0$.  Here $M$ is fixed and
$|\log(\delta/M)|=|\log\delta|+\log M$, so the ratio of the two logarithmic
factors is bounded below by a positive constant depending only on $M$.
Thus restoring $M$ changes only $c$ and
$\delta_0$.

Assume henceforth that $M=1$. Cauchy's estimates on the larger disk and
Taylor truncation give, for every $0<\rho<1/4$, a polynomial map
$P_{\rho}:\C\to\C^2$ such that
\begin{equation}\label{direct:eq:Taylor-bad-map}
 \|F_f-P_{\rho}\|_{C^0(\overline{D(0,1)})}\leq\rho,
 \qquad
 \deg P_{\rho}\leq C_{\epsilon}|\log\rho|.
\end{equation}
Concretely, truncate the Taylor series of $f$ and $f'$ and substitute the
truncations into \eqref{direct:eq:bad-parameter-map}; geometric convergence on
$D(0,1)$ gives \eqref{direct:eq:Taylor-bad-map}.
More explicitly, choose $R=1+\epsilon/2$. Cauchy's estimate bounds the
Taylor coefficients by $CR^{-n}$, so the value remainder after degree $d$
is $O(R^{-d})$ and the derivative remainder is $O(dR^{-d})$ on the unit
disk.  Taking $d\leq C_{\epsilon}|\log\rho|$ sufficiently large makes both
remainders at most $\rho$; multiplication by $u$ in
\eqref{direct:eq:bad-parameter-map} increases the degree by at most one and does
not change this order.

Fix $\ell_{\rm S}=1/2$. For a constant $c_0\in(0,1)$ to be chosen, take
\begin{equation}\label{direct:eq:eta-Sard-proof}
 \eta:=c_0\delta|\log\delta|^{-\ell_{\rm S}}.
\end{equation}
After reducing $\delta_0$, we have $4\eta<\delta$, $2\delta<1$, and
$\eta<1/4$. Choose $P_{\eta}$ from~\eqref{direct:eq:Taylor-bad-map}, and write
$S:=F_f(\overline{D(0,1)})$. Then
\begin{equation}\label{direct:eq:bad-set-inclusion}
 S\subset\overline{N_{\eta}(P_{\eta}(\overline{D(0,1)}))}
 \subset\overline{N_{\eta}(P_{\eta}(\C))}.
\end{equation}
If every $w\in B_{\C^2}(0,\delta)$ satisfied
$\dist(w,S)<3\eta$, then \eqref{direct:eq:bad-set-inclusion} would imply
\begin{equation}\label{direct:eq:bad-ball-covering}
 B_{\C^2}(0,\delta)
 \subset N_{4\eta}(P_{\eta}(\C))
       \cap B_{\C^2}(0,\delta+4\eta).
\end{equation}

Let $d=\max\{1,\deg P_{\eta}\}$. Since
$|\log\eta|\leq3|\log\delta|$ once
$|\log\delta|\geq\max\{1,|\log c_0|\}$,
\eqref{direct:eq:Taylor-bad-map} gives $d\leq C|\log\delta|$.  Apply
\Cref{direct:lem:polynomial-tube} with $r=4\eta$ and
$R=\delta+4\eta<2\delta$.  The four-volume of the right-hand side of
\eqref{direct:eq:bad-ball-covering} is at most
\begin{equation}\label{direct:eq:bad-volume-bound}
 C d(\delta+4\eta)^2(4\eta)^2
 \leq C'c_0^2\delta^4|\log\delta|^{1-2\ell_{\rm S}}
 =C'c_0^2\delta^4.
\end{equation}
The constant $C'$ is independent of $c_0$ under this restriction.
Choose $c_0>0$ small enough that this is strictly smaller than
$\vol_4(B_{\C^2}(0,\delta))=\pi^2\delta^4/2$ (the unit ball in
$\R^4$ has volume $\pi^2/2$).  After decreasing $\delta_0$ to enforce
the preceding scale inequalities, \eqref{direct:eq:bad-ball-covering} is impossible;
some $w\in B(0,\delta)$ has $\dist(w,S)\geq3\eta$.
Equation~\eqref{direct:eq:jet-parameter-distance} now yields
\eqref{direct:eq:strong-local-transversality}, after decreasing the constant $c$.
The point is that the distance is from the whole image
$S=F_f(\overline{D(0,1)})$; hence the same chosen parameter controls the value and
derivative simultaneously at every $|u|\leq1$.
The order of choices is harmless: first fix a positive $c_0$ small enough
for the volume comparison, and then decrease $\delta_0$ so that all
logarithmic degree estimates and the inequalities $4\eta<\delta$ and
$2\delta<1$ hold simultaneously.
\end{proof}

\subsection{The colored net and one-color correction}\label{direct:sec:coloring}

For a separation parameter $D>10$, to be chosen after $A$ below, take a
maximal $r_{\rm cov}$-separated set $\Lambda_N$ in $(\Xsurf,\geo_N)$. Maximality
ensures that its $r_{\rm cov}$-balls cover $\Xsurf$. Join two centers when their
distance is less than $D$. Because $r_{\rm cov}$ is fixed, uniform packing
bounds give a constant $C_{\rm col}$ and the fixed integer
$Q_0=\lceil C_{\rm col}D^2+1\rceil$.  Greedy coloring, padded with empty
classes when necessary, gives
\begin{equation}\label{direct:eq:coloring}
 \Lambda_N=\Lambda_N^1\sqcup\cdots\sqcup\Lambda_N^{Q_0},
 \qquad Q_0=\lceil C_{\rm col}D^2+1\rceil,
\end{equation}
such that same-color centers are at least $D$ apart. Annular packing and
\eqref{direct:eq:peak-decay} imply
\begin{equation}\label{direct:eq:peak-sum}
 \sup_x\sum_{p\in\Lambda_N}\sum_{j=0}^1
 |\nabla^r\sigma^j_{p,N}(x)|\leq C_r,
\qquad 0\leq r\leq3,
\end{equation}
and
\begin{equation}\label{direct:eq:same-color-tail}
 \sup_{p\in\Lambda_N^\alpha}
 \sum_{\substack{q\in\Lambda_N^\alpha\\q\neq p}}
 \sum_{j=0}^1\|\sigma^j_{q,N}\|_{C^2(B_{\geo_N}(p,2r_{\rm cov}))}
 \leq Ce^{-cD}.
\end{equation}
Indeed, every vertex has at most $C_{\rm col}D^2$ neighbors: the disjoint
$\geo_N$-balls of radius $r_{\rm cov}/2$ around those neighbors lie in a ball
of radius $D+r_{\rm cov}/2$, and bounded geometry compares their areas.
The greedy coloring algorithm therefore uses at most
$C_{\rm col}D^2+1$ colors, proving
\eqref{direct:eq:coloring}.
For completeness, compactness and the scaling $\geo_N=N\geo$ give
$\operatorname{Area}_{\geo_N}B_{\geo_N}(x,R)\leq CR^2$ for every $R>0$,
with a constant independent of $N$ and $x$. The disjoint balls of radius
$r_{\rm cov}/2$ have a uniform positive area, so an
$r_{\rm cov}$-separated net has at most $C(R+1)^2$ centers in such a ball.
In particular, a unit-width annulus of radius $k$ contains at most
$C(k+2)^2$ centers. Combining this bound with~\eqref{direct:eq:peak-decay}
sums $\sum_{k\geq0}C(k+2)^2e^{-ck}$ and proves~\eqref{direct:eq:peak-sum}.
For $y\in B_{\geo_N}(p,2r_{\rm cov})$ one has
$d_{\geo_N}(y,q)\geq d_{\geo_N}(p,q)-2r_{\rm cov}$. Same-color separation
therefore gives a tail bounded by
$C\sum_{k\geq\lfloor D\rfloor}(k+2)^2e^{-ck}\leq Ce^{-c'D}$.
This proves~\eqref{direct:eq:same-color-tail}, with constants independent of
$D$ and $N$.

Let $C_{\rm pk}$ be a constant for~\eqref{direct:eq:peak-sum} with $r\leq3$.
The reciprocal estimate in~\Cref{direct:prop:peaks} and the quotient rule give a
fixed number $M_*$ with the following property: if
$\|s\|_{C^3_N}\leq2C_{\rm pk}$, then
$\|s/\sigma^0_{p,N}\|_{C^1(|u|<1+\epsilon_*)}\leq M_*$ for every center.
Apply~\Cref{direct:prop:estimated-sard} once, with $\epsilon=\epsilon_*$ and
$M=M_*$, and denote its constants by $c_{\rm S},\ell_{\rm S},\delta_0$.
This fixes the Sard data before the perturbation recursion is chosen.

Let $K\geq1$ be a uniform constant, supplied by
\eqref{direct:eq:jet-comparison} together with the frame and reciprocal bounds in
\Cref{direct:prop:peaks}, comparing normalized holomorphic first jets with the
invariant norm in~\eqref{direct:eq:eta-transverse}. The notation
$c_{\rm S}$ distinguishes the Sard constant from the off-diagonal decay rate
$c$ in~\eqref{direct:eq:peak-decay}. Suppose the current section $s$ has the
stated $C^3_N$ bound. For one color, order its centers arbitrarily. When
processing $p$, divide the current section by $\sigma^0_{p,N}$ and
apply~\Cref{direct:prop:estimated-sard}. Choose
$w_p=(w_{p,0},w_{p,1})$, $|w_p|\leq\delta$, and update
\[
 s\longleftarrow s+w_{p,0}\sigma^0_{p,N}
                    +w_{p,1}\sigma^1_{p,N}.
\]
The recursion below proves $\delta_j\leq2^{1-j}\delta_1$ and
$\delta_1\leq1$. Consequently, \eqref{direct:eq:peak-sum} gives, at every partial
stage of the construction,
\[
 \|s\|_{C^3_N}\leq C_{\rm pk}\sum_{j\geq1}\delta_j
 \leq2C_{\rm pk}.
\]
This verifies inductively the bound used to choose $M_*$, so there is no
bootstrap or parameter-order circularity. By
\eqref{direct:eq:peak-scale-containments}, the enlarged disk
$|u|<1+\epsilon_*$ required in \Cref{direct:prop:estimated-sard} lies in the
zero-free coordinate disk on which the reciprocal of
$\sigma^0_{p,N}$ is controlled.
In the normalized local function, the actual perturbation is
\[
 w_{p,0}+w_{p,1}u+w_{p,1}\varepsilon_{p,N}(u).
\]
Thus, \eqref{direct:eq:peak-ratio-model} contributes at most
$C_{\rm err}|w_{p,1}|N^{-1/2}\leq C_{\rm err}\delta N^{-1/2}$ to its $C^1$
norm, for a fixed constant $C_{\rm err}$. For all $N$
large enough that this is at most
$\frac{c_{\rm S}}{10}\delta|\log\delta|^{-\ell_{\rm S}}$, the normalized lower
bound is at least
$\frac{9c_{\rm S}}{10}\delta|\log\delta|^{-\ell_{\rm S}}$. After conversion to
the invariant norm, we retain the simpler lower bound
\begin{equation}\label{direct:eq:local-eta}
 \widehat\eta(\delta):=\frac{c_{\rm S}}{2K}\delta|\log\delta|^{-\ell_{\rm S}}.
\end{equation}
Previously chosen peaks are included in the current function. Later
same-color choices change its invariant first jet by at most
$C_{\rm tail}\delta e^{-cD}$. More precisely, same-color separation is at
least $D$, and
\eqref{direct:eq:same-color-tail} shows that summing over all centers treated later
still gives this exponential bound, rather than a bound proportional to the
number of centers. The final correction for the color is therefore
$\widehat\eta(\delta)$-transverse
on every one of its covering $r_{\rm cov}$-balls, provided
\begin{equation}\label{direct:eq:tail-condition}
 C_{\rm tail}e^{-cD}
 \leq\frac{c_{\rm S}}{10K}|\log\delta|^{-\ell_{\rm S}}.
\end{equation}
Indeed, before the tail is added the invariant lower bound is at least
\[
 \frac{9c_{\rm S}}{10K}\,
 \delta|\log\delta|^{-\ell_{\rm S}}.
\]
Condition \eqref{direct:eq:tail-condition} spends at most another
$c_{\rm S}\delta|\log\delta|^{-\ell_{\rm S}}/(10K)$ and therefore leaves
\[
 \frac{4c_{\rm S}}{5K}\,
 \delta|\log\delta|^{-\ell_{\rm S}}
 >\widehat\eta(\delta).
\]
Consequently every local lower bound obtained during the sequential
procedure survives to the end of that color.
Its global $C^3_N$ norm is at most $C\delta$ by \eqref{direct:eq:peak-sum}.

\subsection{Induction without circular dependence}\label{direct:sec:induction}

The recursion is formulated in terms of the Sard exponent $\ell_{\rm S}$ from
\Cref{direct:prop:estimated-sard} and is valid for every fixed finite nonnegative value of that
exponent. The sharper tube estimate permits
$\ell_{\rm S}=1/2$ by~\eqref{direct:eq:Sard-exponent-half} and therefore strengthens
the local lower bound. The induction below closes for every fixed finite nonnegative value
of $\ell_{\rm S}$.

The cutoff constant $A_{\rm pk}$ and the peak sections are now fixed.
Choose a separate constant $A$ sufficiently large. The recursion is
\begin{equation}\label{direct:eq:strong-recursion}
 \eta_0:=1,\qquad
 \delta_j:=\eta_{j-1}/A,\qquad
 \eta_j:=c_1\delta_j|\log\delta_j|^{-\ell_{\rm S}},
 \quad j\geq1,
\end{equation}
where we take
\[
 c_1:=\frac{c_{\rm S}}{2K}.
\]
It is enough to take
\begin{equation}\label{direct:eq:A-choice}
 A\geq\max\{4C,2c_1,2e,2\delta_0^{-1}\}.
\end{equation}
Here $\delta_0$ is the smallness threshold in
\Cref{direct:prop:estimated-sard}, for the single value $M_*$ fixed before $A$.
Indeed, $\delta_1=1/A<\min\{e^{-1},\delta_0\}$ and
\[
 C\delta_j=\frac{C}{A}\eta_{j-1}\leq\frac14\eta_{j-1}.
\]
Moreover, $\delta_j<e^{-1}$ and
\[
 \frac{\eta_j}{\eta_{j-1}}
 =\frac{c_1}{A}|\log\delta_j|^{-\ell_{\rm S}}\leq\frac12.
\]
It follows that
$\delta_{j+1}/\delta_j=\eta_j/\eta_{j-1}\leq1/2$. This is the geometric
decay used both in the global norm budget and in the uniform validation of
the Sard parameter $M_*$.
The induction invariant is the following: immediately before color $j$ is
processed, the current section is at least $\eta_{j-1}$-transverse on every
covering $r_{\rm cov}$-ball belonging to colors $1,\ldots,j-1$. At step $j$,
process color $j$ with size $\delta_j$. The new perturbation
leaves at least $\eta_{j-1}/2\geq\eta_j$ on all previously treated balls
and creates the lower bound $\eta_j$ on the new covering balls. Thus, after
step $j$, the section is $\eta_j$-transverse on all balls belonging to the
first $j$ colors. If the padded color class is empty, the section is
unchanged and the same conclusion follows directly from
$\eta_j\leq\eta_{j-1}$.

It remains to justify the choice of $D$.  Taking logarithms in
\eqref{direct:eq:strong-recursion} gives, by induction,
\begin{equation}\label{direct:eq:eta-log-bound}
 |\log\eta_j|\leq Cj\log(j+2).
\end{equation}
Here is the elementary induction. Write $x_j=|\log\eta_j|$. Since
$\delta_j=\eta_{j-1}/A<1$, the recursion gives
\[
 x_j=x_{j-1}+\log(A/c_1)
      +\ell_{\rm S}\log(x_{j-1}+\log A).
\]
If $x_{j-1}\leq C(j-1)\log(j+1)$, the last logarithm is
$O(\log(j+1))$; choosing $C$ larger than the fixed implicit constant closes
the induction and proves \eqref{direct:eq:eta-log-bound}.
After enlarging the fixed constant, $Q_0\leq C'D^2$. The smallest right-hand
side of~\eqref{direct:eq:tail-condition} is therefore bounded below by a negative
power of $D^2\log(D+2)$. More explicitly, for $j\leq Q_0$,
\[
 |\log\delta_j|=|\log\eta_{j-1}|+\log A
 \leq C D^2\log(D+2),
\]
so the right-hand side of \eqref{direct:eq:tail-condition} is at least
$c'[D^2\log(D+2)]^{-\ell_{\rm S}}$. The exponential $e^{-cD}$ is eventually
smaller than this polynomial-logarithmic quantity. Hence a sufficiently large
fixed $D$ makes
\eqref{direct:eq:tail-condition} valid for every $1\leq j\leq Q_0$, without any
circular dependence and independently of $N$. Equivalently,
$e^{-cD}[D^2\log(D+2)]^{\ell_{\rm S}}\to0$ as $D\to\infty$. This explicit
limit is the only asymptotic comparison required for the present scheme.
The sequence $\delta_j$
decreases geometrically, so \eqref{direct:eq:peak-sum} also gives a uniform
$C^3_N$ bound for all intermediate and final sections.

After $D$ has been fixed, the number $Q_0$ and the sequences
$\{\delta_j\}_{j=1}^{Q_0}$ and $\{\eta_j\}_{j=1}^{Q_0}$ form a
finite collection of positive constants independent of $N$.  We may
therefore choose one $N_0$ such that, for every $N\geq N_0$ and every
$1\leq j\leq Q_0$,
\begin{equation}\label{direct:eq:uniform-N-error}
 C_{\rm err}\delta_jN^{-1/2}
 \leq\frac{c_{\rm S}}{10}\delta_j|\log\delta_j|^{-\ell_{\rm S}}.
\end{equation}
This single choice controls all $O(N^{-1/2})$ errors in the induction.

The dependency order is now explicit: first fix the peak and chart constants,
then $M_*$ and the resulting Sard constants; next choose $A$ in
\eqref{direct:eq:A-choice}; then choose the single separation constant $D$, which
also fixes $Q_0$ and the finite sequences
$\{\delta_j\},\{\eta_j\}$; finally choose $N_0$.  No constant chosen earlier
depends on $N$, and no later choice feeds back into $A$ or the local Sard
lemma.

We can now prove the theorem stated in the introduction.

\begin{proof}[Proof of \Cref{direct:thm:direct-curve-transversality}]
Choose $D$ as above, color the net, initialize $s^{(0)}=0$, and process its
$Q_0$ (possibly padded) colors according to~\eqref{direct:eq:strong-recursion}. The
invariant before the first color is vacuous, and the initial $C^3_N$ norm is
zero. At every
later stage, \eqref{direct:eq:peak-sum} and the geometric decay of $\delta_j$ give
a uniform $C^3_N$ bound, so the uniform Sard parameter discussed above
remains valid. Applying the induction invariant stated in
\Cref{direct:sec:induction} through color $Q_0$ shows that the resulting section
is $\eta_{Q_0}$-transverse on every ball of the colored net. Those
$r_{\rm cov}$-balls cover $\Xsurf$, so the section is globally
$\eta_{Q_0}$-transverse. The peak-sum
estimate and the geometric decay of the perturbation sizes give the asserted
derivative bounds. Take $\eta=\eta_{Q_0}$.
\end{proof}

The estimates through order two are those used by the packing criterion
in~\Cref{prop:transversality-to-packing}; the full estimates through order three
are those required by the uniform transverse-family hypothesis
in~\cite{PaperC}. The $C^3_N$ bound also controls the quotient estimates
used for the inductive Sard parameter $M_*$. Thus
\Cref{direct:thm:direct-curve-transversality} supplies the uniform bounds required
by both applications.

\begin{remark}[The role of the two parameters, and what is imported]
\label{direct:rem:why-two-parameters}
Mohsen's integrable version of Donaldson's theorem
\cite[Theorem~2.2 and its proof]{MohsenLimit} independently gives the conclusion
of~\Cref{direct:thm:direct-curve-transversality}. A reader who prefers to quote that
theorem may stop there; the packing application below and the results of~\cite{PaperC} depend only on the
statement, never on the internal details of either proof.

What the construction above adds is an explicit account of the arithmetic that
makes this particular two-parameter scheme close in complex dimension one. The
mechanism is visible in~\eqref{direct:eq:strong-local-transversality}: under a
uniformly invertible linear map, the norm of the full scalar first jet is
comparable to the distance from $(w_0,w_1)$ to the image
$F_f(\overline{D(0,1)})\subset\C^2$. The tube estimate gives the
logarithmic factor $|\log\delta|^{-1/2}$ of
\eqref{direct:eq:Sard-exponent-half}. After $Q_0$ colors,
\eqref{direct:eq:eta-log-bound} bounds the accumulated logarithm by
$O(D^2\log D)$, and \eqref{direct:eq:tail-condition} reduces to
$e^{-cD}\lesssim[D^2\log(D+2)]^{-\ell_{\rm S}}$, which holds for all large $D$.
More generally, the same comparison works for every fixed finite nonnegative value of
$\ell_{\rm S}$, exactly as stated at the start of~\Cref{direct:sec:induction}. Thus the
linear tube bound identifies the exponent $1/2$ for this argument; it is not
claimed to rule out other one-parameter or multi-parameter constructions.

The construction uses two complex parameters at each center and therefore avoids
the full Donaldson--Auroux asymptotically holomorphic machinery. Its only
external analytic existence theorem is the weighted $L^2$ theorem for
$\bar\partial$ used in \Cref{direct:lem:appendix-weighted-dbar}; in particular, no
Bergman-kernel asymptotics and no real-algebraic tube theorem are required.
\end{remark}

\section{A quantitative local zero lemma}\label{sec:local-lemma}

\subsection{A uniform model for simple zeros}

We now prove the local analytic result that turns
\eqref{eq:our-uniform-estimates} into controlled zero geometry.  It is useful
to state it for ordinary holomorphic functions.

The following local lemma converts transversality into quantitative information about zeros.

\begin{lemma}\label{lem:local-zero}
For every $M\geq1$ and $\eta>0$ there are constants
$r=r(M,\eta)>0$, $a=a(M,\eta)>0$, and $A=A(M)>0$ with the following
property.  Let $f$ be holomorphic on the Euclidean disk $D(0,2)$ and assume
\begin{equation}\label{eq:local-assumptions}
 \sup_{D(0,2)}\bigl(|f|+|f'|+|f''|\bigr)\leq M,
 \qquad |f(z)|+|f'(z)|\geq\eta\quad(z\in D(0,2)).
\end{equation}
Then:
\begin{enumerate}[label=\textup{(\roman*)}]
\item every zero in $D(0,1)$ is simple;
\item distinct zeros in $D(0,1)$ have distance at least $2r$;
\item if $p\in D(0,1)$ is a zero and $|z-p|<r$, then
\[
 a|z-p|\leq|f(z)|\leq A|z-p|;
\]
\item if $z\in D(0,1)$ is at distance at least $r$ from the zero set of
$f$, then $|f(z)|\geq ar/4$.
\end{enumerate}
\end{lemma}

\begin{proof}
At a zero $p$, \eqref{eq:local-assumptions} gives $|f'(p)|\geq\eta$, proving
(i).  Taylor's formula and the bound on $f''$ give, whenever the segment is
contained in $D(0,2)$,
\begin{equation}\label{eq:taylor-zero}
 f(z)=f'(p)(z-p)+R_p(z),
 \qquad |R_p(z)|\leq\frac{M}{2}|z-p|^2.
\end{equation}
Choose
\[
 0<r<\min\left\{\frac14,\frac{\eta}{4M}\right\}.
\]
For $|z-p|<2r$, \eqref{eq:taylor-zero} gives
\[
 \frac{\eta}{2}|z-p|\leq|f(z)|\leq\frac{3M}{2}|z-p|.
\]
This proves (iii), with $a=\eta/2$ and $A=3M/2$.  It also proves (ii): a
second zero $q$ with $0<|q-p|<2r$ would contradict the lower bound.

It remains to prove (iv), including the case where $f(z)$ is small but $z$
is not already known to be near a zero. If $f(z)=0$, the conclusion
that $z$ lies near a zero is immediate, so in the following Rouch\'e
argument we assume $f(z)\ne0$. Write $b=f'(z)$. Suppose that
$|f(z)|<\eta/2$.  Then $|b|\geq\eta/2$.  Set
\[
 \rho:=\frac{4|f(z)|}{|b|}.
\]
If $|f(z)|$ is small enough that $\rho<r$, compare on $|\zeta-z|=\rho$
the affine function
\[
 \ell(\zeta)=f(z)+b(\zeta-z)
\]
with $f$.  On that circle,
\[
 |\ell(\zeta)|\geq |b|\rho-|f(z)|=3|f(z)|,
\]
whereas Taylor's theorem gives
\[
 |f(\zeta)-\ell(\zeta)|\leq\frac M2\rho^2.
\]
Since $|b|\geq\eta/2$, one has
$\rho\leq8|f(z)|/\eta$ and hence
\[
 \frac M2\rho^2\leq\frac{32M}{\eta^2}|f(z)|^2.
\]
It is therefore enough to impose
\[
 |f(z)|<c_0:=\min\left\{\frac\eta2,
       \frac{\eta r}{16},\frac{3\eta^2}{64M}\right\}.
\]
Then $\rho<r$ and the quadratic error is strictly smaller than
$3|f(z)|$.  Rouch\'e's theorem shows that $f$ has a zero in $D(z,\rho)$, so
\begin{equation}\label{eq:small-implies-zero}
 |f(z)|<c_0\quad\Longrightarrow\quad
 \dist(z,Z(f))\leq \frac{8}{\eta}|f(z)|.
\end{equation}
The choice of $c_0$ gives $8c_0/\eta<r$.  Hence the contrapositive of
\eqref{eq:small-implies-zero} gives $|f(z)|\geq c_0$ whenever
$\dist(z,Z(f))\geq r$.  Replacing the previous value of $a$ by
$\min\{\eta/2,4c_0/r\}$ preserves (iii) and makes $c_0\geq ar/4$, proving
(iv).
\end{proof}

\begin{remark}
The final part of the preceding proof is included to avoid an often-used but
insufficient assertion that a large derivative by itself forces a nearby
zero.  The quantitative Rouch\'e argument is what makes the assertion valid.
\end{remark}

\section{Geometry of the transverse zero divisors}
\label{sec:zero-geometry}

\subsection{Separation and local norm estimates}

Let $s_N$ be the sections given by \Cref{thm:curve-transversality}, with
\[
 D_N=Z_{s_N}=w_{1,N}+\cdots+w_{N,N}.
\]
The degree is exactly $N$, because a nonzero holomorphic section of $L^N$
has a zero divisor of degree $\deg(L^N)=N$.  We transfer
Lemma \ref{lem:local-zero} to uniformly controlled rescaled holomorphic
coordinates and frames.

Because $\Xsurf$ is compact and $\heo,\geo$ are smooth, there is a fixed $R>2$
with the following property. Let $N$ be sufficiently large and
$x\in \Xsurf$. Then the $\geo_N$-ball $B_{\geo_N}(x,R)$ carries a holomorphic
coordinate $u$ and a holomorphic frame for $L^N$ with the following
three properties.
\begin{itemize}
\item the coordinate metric and its first two derivatives are uniformly
comparable to the Euclidean metric;
\item the frame metric and the Chern connection coefficients have uniform
$C^2$ bounds on fixed subdisks;
\item the estimates in \eqref{eq:our-uniform-estimates} imply the hypotheses
of Lemma \ref{lem:local-zero}, with constants independent of $x$ and $N$.
\end{itemize}
These are elementary consequences of Taylor expansion in a fixed finite
atlas after the change of scale $u=\sqrt N\,z$.  The connection coefficient
of $\heo^N$ remains bounded at that scale because a normal frame makes the
linear term vanish and the curvature term is quadratic.
Concretely, if $s_N=f_Ne_N$ in the rescaled holomorphic frame, then
$\nabla^{1,0}s_N=(f_N'+A_Nf_N)e_N$, with $A_N$ uniformly bounded on a
fixed disk.  Differentiating once more expresses $f_N''$ through
$\nabla^2s_N$, $f_N'$, and uniformly bounded coefficients.  The frame norm
is bounded above and below there.  Hence the invariant $C^2_N$ estimate
gives a uniform bound for $|f_N|+|f_N'|+|f_N''|$, while
\eqref{eq:eta-transverse} gives a comparable lower bound for
$|f_N|+|f_N'|$.  This is the precise passage to
Lemma \ref{lem:local-zero}; it is also recorded in
\eqref{eq:jet-comparison}--\eqref{eq:C2-comparison} of
Section \ref{app:rescaling}.

A simple zero has a uniformly nonvanishing first derivative in rescaled coordinates.  The local zero lemma therefore prevents two zeros from approaching one another at a smaller scale.

\begin{proposition}
\label{prop:separation}
There exists $r_0>0$, independent of $N$, such that every zero of $s_N$ is
simple and
\[
\dist_{\geo_N}(w_{i,N},w_{j,N})\geq 2r_0\qquad(i\neq j).
\]
Equivalently,
\[
\dist_{\geo}(w_{i,N},w_{j,N})\geq 2r_0N^{-1/2}.
\]
\end{proposition}

\begin{proof}
In each rescaled coordinate chart the representing holomorphic function
satisfies Lemma \ref{lem:local-zero} with the same $M$ and $\eta$
(after changing them by uniform frame constants).  Parts (i) and (ii) of
that lemma give the claim.
\end{proof}

The same local analysis controls the section norm by the distance to the divisor.  Away from the zeros, quantitative transversality supplies the complementary positive lower bound.

\begin{proposition}
\label{prop:norm-model}
There are constants $c,C,r_0>0$, independent of $N$, such that
\begin{align}
 c\,\dist_{\geo_N}(z,D_N)&\leq |s_N(z)|_{\heo^N}
 \leq C\,\dist_{\geo_N}(z,D_N),
 &&\dist_{\geo_N}(z,D_N)<r_0,\label{eq:near-zero-model}\\
 c&\leq |s_N(z)|_{\heo^N}\leq C,
 &&\dist_{\geo_N}(z,D_N)\geq r_0.\label{eq:away-zero-model}
\end{align}
Here $\dist_{\geo_N}(z,D_N)=\min_j\dist_{\geo_N}(z,w_{j,N})$.
\end{proposition}

\begin{proof}
The upper bound $|s_N|\leq C$ is part of \eqref{eq:our-uniform-estimates}.
If $z$ is within $r_0$ of a zero, separation makes that zero unique and
part (iii) of Lemma \ref{lem:local-zero}, together with uniform coordinate and
frame comparison, gives \eqref{eq:near-zero-model}.  Part (iv) gives the
lower bound in \eqref{eq:away-zero-model}.  All comparison constants are
uniform by compactness and the rescaling described above.
To spell out the last step, center a controlled rescaled chart at $z$.
If $\dist_{\geo_N}(z,D_N)\geq r_0$, the representing function has no zero in the
corresponding Euclidean disk of radius comparable to $r_0$.  After reducing
$r_0$ once to absorb the metric comparison constants, Lemma 
\ref{lem:local-zero}\textup{(iv)} gives a fixed positive lower bound at
the chart center.  Multiplication by the uniformly nondegenerate frame norm
then gives the invariant lower bound for $|s_N(z)|_{\heo^N}$.
\end{proof}

\subsection{Covering at the natural scale}

The uniform norm estimates also control the largest zero-free regions.
The following result complements the separation estimate by establishing a covering bound at the same geometric scale.

The curvature argument below complements separation by establishing covering at the natural scale.

\begin{theorem}
\label{thm:transverse-covering}
There is a constant $R_{\mathrm{cov}}<\infty$, independent of $N$, such
that the divisors $D_N=Z_{s_N}$ constructed in
Theorem \ref{thm:curve-transversality} satisfy
\begin{equation}\label{eq:transverse-covering}
 \sup_{x\in \Xsurf}\dist_{\geo_N}(x,D_N)\leq R_{\mathrm{cov}}
\end{equation}
for every sufficiently large $N$.  Equivalently, their covering radius in
any fixed unscaled metric is $O(N^{-1/2})$.
\end{theorem}

\begin{proof}
The idea is that a large zero-free ball is impossible.  On such a ball the
Poincar\'e--Lelong identity makes $\log|s_N|$ superharmonic up to a fixed
quadratic correction, so its value at the center exceeds its boundary average by
an amount growing with the radius.  But the norm model bounds both by the same
constant.

Write $F_N=\log|s_N|_{\heo^N}$.  The two-sided estimate
\eqref{eq:away-zero-model} gives a constant $C_1$ such that
\begin{equation}\label{eq:potential-away-zero-two-sided}
 |F_N(y)|\leq C_1
 \qquad\text{whenever }\dist_{\geo_N}(y,D_N)\geq r_0.
\end{equation}

We show that a zero-free rescaled ball cannot have arbitrarily large fixed
radius.  Fix a number $A>0$, to be chosen below.  In the rescaled coordinate
$u=\sqrt N z$ of Section \ref{app:rescaling}, the metric $\geo_N$ is uniformly
comparable with the Euclidean metric on $|u|\leq A$, and
\eqref{eq:measure-scaling} reads
\[
 N\,\diff\pmeas=m_x(u/\sqrt N)\,dA(u),
 \qquad m_x(u/\sqrt N)\geq m_0>0.
\]
If this coordinate disk contains no zero, then the Poincar\'e--Lelong
identity gives $\ddc F_N=-N\pmeas$ there.  With the convention fixed in
\eqref{eq:green-normalization}, and with
$\Delta_u=\partial_{\Re u}^2+\partial_{\Im u}^2$ the ordinary Euclidean
Laplacian, this implies
\begin{equation}\label{eq:covering-laplacian}
 \Delta_u F_N=-2\pi m_x(u/\sqrt N)\leq-\kappa,
 \qquad \kappa:=2\pi m_0>0.
\end{equation}
Thus $F_N+(\kappa/4)|u|^2$ is superharmonic.  The mean-value inequality on
the circle $|u|=A$ gives
\begin{equation}\label{eq:covering-mean-value}
 F_N(0)-\frac1{2\pi}\int_0^{2\pi}F_N(A\e^{\imag\theta})\,\diff\theta
 \geq \frac{\kappa A^2}{4}.
\end{equation}

Choose $A$ so large that $\kappa A^2/4>2C_1$.  This is a fixed choice made
before $N$; by Section \ref{app:rescaling}, the required coordinate disk exists
uniformly for all sufficiently large $N$.  Let $C_{\geo}$ be the uniform upper
comparison between Euclidean distance in this disk and $\dist_{\geo_N}$.  If some
$x\in \Xsurf$ satisfied
\[
 \dist_{\geo_N}(x,D_N)>r_0+C_{\geo}A,
\]
then every point of the coordinate disk $|u|\leq A$ centered at $x$ would
have $\geo_N$-distance at least $r_0$ from $D_N$.  Estimate
\eqref{eq:potential-away-zero-two-sided} would bound both terms on the
left-hand side of \eqref{eq:covering-mean-value} in absolute value by
$C_1$, contradicting the choice of $A$.  Therefore
$R_{\mathrm{cov}}:=r_0+C_{\geo}A$ proves \eqref{eq:transverse-covering}.
The equivalence with the unscaled estimate follows from $\geo_N=N\geo$ and the
bi-Lipschitz equivalence of fixed metrics on compact $\Xsurf$.
Increasing $R_{\mathrm{cov}}$ once makes the covering bound strict; this
choice is used for the open outer balls below.
\end{proof}

Separation and covering together control the inner and outer radii of the Voronoi cells, both measured in the rescaled metric.

\begin{corollary}
\label{cor:transverse-voronoi}
For a zero $w_{j,N}\in D_N$, let
\[
 \mathcal V_{j,N}:=
 \{x\in \Xsurf:\dist_{\geo_N}(x,w_{j,N})\leq
              \dist_{\geo_N}(x,w_{k,N})\text{ for every }k\}
\]
be its closed $\geo_N$-Voronoi cell.  After reducing the separation constant
$r_0$ if necessary, uniformly in $j$ and all sufficiently large $N$,
\begin{equation}\label{eq:voronoi-two-sided-radius}
 B_{\geo_N}(w_{j,N},r_0)\subset\mathcal V_{j,N}
 \subset B_{\geo_N}(w_{j,N},R_{\mathrm{cov}}).
\end{equation}
Consequently there are constants $c,C>0$ such that
\begin{equation}\label{eq:voronoi-area-diameter}
 \frac{c}{N}\leq\pmeas(\mathcal V_{j,N})\leq\frac{C}{N},
 \qquad
 \operatorname{diam}_{\geo_N}(\mathcal V_{j,N})
 \leq2R_{\mathrm{cov}}.
\end{equation}
\end{corollary}

\begin{proof}
The separation estimate in \Cref{prop:separation} says that every other
zero is at $\geo_N$-distance at least $2r_0$ from $w_{j,N}$.  The triangle
inequality therefore places $B_{\geo_N}(w_{j,N},r_0)$ inside the closed Voronoi
cell.  If $x\in\mathcal V_{j,N}$, then $w_{j,N}$ is a nearest zero to $x$;
\Cref{thm:transverse-covering} gives
$\dist_{\geo_N}(x,w_{j,N})=\dist_{\geo_N}(x,D_N)<R_{\mathrm{cov}}$, by the strict choice
of the covering constant.  This proves
\eqref{eq:voronoi-two-sided-radius} and the diameter bound.

For each fixed radius, bounded rescaled geometry and
\eqref{eq:measure-scaling} give
\[
 \frac{c_r}{N}\leq\pmeas(B_{\geo_N}(x,r))\leq\frac{C_r}{N}
\]
uniformly in the center and in sufficiently large $N$.  Apply this estimate
to the two balls in \eqref{eq:voronoi-two-sided-radius}.  Possible overlaps
of closed Voronoi cells occur only on their boundaries and do not affect
the individual measure estimates.
\end{proof}

\section{Uniform logarithmic integrability}\label{sec:log-integrability}

\subsection{Logarithmic moments of the section norm}

The crucial estimate is now elementary but must retain the correct scaling.

\begin{theorem}\label{thm:log-L1}
For every fixed $q>0$ there is a constant $C_q<\infty$, independent of $N$, such that
\begin{equation}\label{eq:log-L1}
 \int_{\Xsurf}\left|\log|s_N|_{\heo^N}\right|^q \diff\pmeas\leq C_q.
\end{equation}
\end{theorem}

\begin{proof}
We split $\Xsurf$ into the microscopic balls around the zeros and their
complement.
On the complement the norm model bounds $|\log|s_N||$ uniformly; on each ball
the model reduces the estimate to the integrability of $|\log|u||^q$ for each fixed $q>0$, with contribution $O(N^{-1})$.  Since there are exactly
$N$ balls, the total is $O(1)$.

Let $0<r<r_0$ be so small that the $\geo_N$-balls
\[
 B_{j,N}:=B_{\geo_N}(w_{j,N},r),\qquad j=1,\ldots,N,
\]
are pairwise disjoint, by Proposition \ref{prop:separation}.  On the
complement of their union, the two estimates in Proposition
\ref{prop:norm-model} bound
$|\log|s_N||$ uniformly: use \eqref{eq:near-zero-model} when the distance
to $D_N$ lies between $r$ and $r_0$, and \eqref{eq:away-zero-model} beyond
$r_0$.

On $B_{j,N}$, use a rescaled coordinate $u$ centered at $w_{j,N}$.
Uniform bounded geometry and \eqref{eq:near-zero-model} give
\[
 \left|\log|s_N(z)|_{\heo^N}\right|
 \leq C+|\log|u||.
\]
The original probability form scales as
\[
 \diff\pmeas(z)\leq \frac{C}{N}\,dA(u)
\]
on a fixed $u$-disk $|u|<R$ containing the coordinate image of $B_{j,N}$.
Therefore, for each fixed $q>0$,
\[
 \int_{B_{j,N}}\left|\log|s_N|_{\heo^N}\right|^q \diff\pmeas
 \leq \frac{C_q}{N}\int_{|u|<R}(1+|\log|u||^q)\diff A(u)
 \leq\frac{C_q'}{N}.
\]
The planar integral is finite because polar coordinates give
\[
 \int_{|u|<R}|\log|u||^q\,\diff A(u)
 =2\pi\int_0^R\rho|\log\rho|^q\,\diff\rho<\infty.
\]
Summing over the exactly $N$ zeros gives a uniform bound.  Adding the
bounded contribution from the complement proves \eqref{eq:log-L1}.
\end{proof}

\subsection{Moment bounds and the packing criterion}

\begin{corollary}\label{cor:Lbeta}
For every $\beta>0$ there is $C_{\beta}$ independent of $N$ such that
\[
 \int_{\Xsurf}|s_N|_{\heo^N}^{\beta}\diff\pmeas\leq C_{\beta}.
\]
\end{corollary}

\begin{proof}
The $C^0$ part of \eqref{eq:our-uniform-estimates} gives
$|s_N|_{\heo^N}\leq C_0$ on $\Xsurf$, and $\pmeas(\Xsurf)=1$.
\end{proof}

We now isolate the properties of the holomorphic construction used in the
preceding three sections, since nothing else about it enters the proof of
\Cref{thm:main}.  The
answer is exactly the two displayed hypotheses of
\eqref{eq:our-uniform-estimates}.

These estimates isolate a criterion that applies to any family of holomorphic sections satisfying uniform rescaled bounds and transversality, regardless of how the sections are constructed.

\begin{proposition}
\label{prop:transversality-to-packing}
Let $t_N\in H^0(\Xsurf,L^N)$ be any family for which
$\|t_N\|_{C^2_N}\leq M$ and
\[
 |t_N(x)|_{\heo^N}+|\nabla t_N(x)|_{\geo_N,\heo^N}\geq\eta
 \qquad(x\in \Xsurf)
\]
 with $M$ and $\eta$ independent of $N$.  Then the zero configurations
$Z_{t_N}$ satisfy the exponent-free pointwise estimate
\begin{equation}\label{eq:transversality-pointwise-packing}
 \sup_{z\in \Xsurf}\sum_{w\in Z_{t_N}}U_{\Xsurf}(z,w)\leq C,
\end{equation}
where $C$ is independent of $N$.  In particular, for every fixed
$\beta>0$,
\[
 \sup_N \calB_{\beta}(Z_{t_N};\Xsurf)<\infty.
\]
\end{proposition}

\begin{proof}
The proofs of Propositions \ref{prop:separation} and \ref{prop:norm-model}
and Theorem \ref{thm:log-L1}
use only the two displayed hypotheses and bounded rescaled geometry.
They therefore apply to $t_N$ with constants depending only on $M$, $\eta$,
and the fixed geometric data.  In particular,
$\log|t_N|_{\heo^N}\leq\log M$ and
$\left|\int_{\Xsurf}\log|t_N|_{\heo^N}\,\diff\pmeas\right|\leq C$.  The potential
identity therefore gives \eqref{eq:transversality-pointwise-packing}.
It also gives
\[
\calB_{\beta}(Z_{t_N};\Xsurf)=
\exp\Big(-\beta\int_{\Xsurf}\log|t_N|_{\heo^N}\,\diff\pmeas\Big)\,
 \int_{\Xsurf}|t_N|_{\heo^N}^\beta\,\diff\pmeas,
\]
whose two factors are uniformly bounded exactly as above.
\end{proof}

\begin{remark}
\Cref{prop:transversality-to-packing} isolates the only information about the
holomorphic sections that Route A uses.  Any future method producing
the same uniform rescaled $C^2$ and transversality bounds immediately gives
the Green-packing conclusion, without changing the potential-theoretic
part of the proof.
\end{remark}

\section{Proof of uniform Green packing}\label{sec:main-proof}

\subsection{Completion of the holomorphic construction}

All the pieces of Route A are now in place, and it remains only to assemble
them. The potential identity converts the logarithmic norm of the section
into the Green potential of its divisor; the $C^0$ bound controls the
potential from above; and Theorem \ref{thm:log-L1} controls the mean-zero
correction.

\begin{proof}[Proof of Theorem \ref{thm:main}]
For sufficiently large $N$, use the divisor $D_N=Z_{s_N}$.  By
Proposition \ref{prop:potential-identity},
\[
 \sum_{j=1}^{N}U_{\Xsurf}(z,w_{j,N})
 =\log|s_N(z)|_{\heo^N}-m_N,
 \qquad
 m_N:=\int_{\Xsurf}\log|s_N|_{\heo^N}\diff\pmeas.
\]
The uniform $C^0$ bound on $s_N$ and Theorem \ref{thm:log-L1} imply
\[
 \sup_{\Xsurf}\sum_{j=1}^{N}U_{\Xsurf}(\,\cdot\,,w_{j,N})
 \leq \log C_0+|m_N|\leq C.
\]
Thus \eqref{eq:main-pointwise-packing} holds for every sufficiently large
$N$, with a constant independent of $\beta$.  Consequently, the pointwise estimate immediately yields
\[
 \calB_{\beta}(D_N;\Xsurf)
 =\int_{\Xsurf}\exp\!\left(\beta\sum_{j=1}^N
 U_{\Xsurf}(z,w_{j,N})\right)\diff\pmeas(z)
 \leq \exp(\beta C)
\]
for every $\beta>0$, with the same configuration $D_N$ and the same
constant $C$.  The estimate for $m_N$ uses both signs of
$\log|s_N|$, as required by the potential identity.

Only finitely many positive integers remain.  For each such $N$, choose any
configuration of $N$ points.  Since $U_{\Xsurf}(\cdot,w)$ is bounded above and has
only a negative logarithmic singularity at $w$, the corresponding potential
has finite supremum.  Enlarging the pointwise constant by the maximum of
these finitely many suprema proves \eqref{eq:main-pointwise-packing} for all
$N\geq1$.  Exponentiating and integrating gives
$\calB_{\beta}(D_N;\Xsurf)\leq \exp(\beta C_{\Xsurf,\pmeas})$ simultaneously for all
$\beta>0$.

For the lower bound, Jensen's inequality and the mean-zero normalization of
the Green kernel give, for every $\Wset=(w_1,\ldots,w_N)$,
\[
 \calB_{\beta}(\Wset;\Xsurf)
 \geq\exp\!\left(\beta\sum_{j=1}^{N}
 \int_{\Xsurf} U_{\Xsurf}(z,w_j)\diff\pmeas(z)\right)=1.
\]
Taking the infimum over $\Wset$ completes the proof.
\end{proof}

The potential identity shows how uniform control of sections and their logarithmic means produces the packing bound.

\begin{remark}
The route just completed is the chain
\[
 \begin{aligned}
  &\text{zero-section transversality}\\[-2pt]
  &\qquad\Longrightarrow \text{uniform local linear model}\\[-2pt]
  &\qquad\Longrightarrow \text{simple separated zeros}\\[-2pt]
  &\qquad\Longrightarrow \text{uniform logarithmic integrability}.
 \end{aligned}
\]
The first estimate forces $|s_N|$ to look like the distance to one simple
zero on every microscopic zero ball and prevents two such balls from
overlapping.  Since each $\geo_N$-ball has $\pmeas$-area $O(N^{-1})$,
summing the integrable model $\log|u|$ over the $N$ zero balls costs
only $O(1)$.
Away from those balls the section has a uniform positive lower bound.
Thus zero-section transversality is precisely the condition that converts
the potentially singular integral of $\log|s_N|$ into a uniform estimate.
\end{remark}

\section{A spectral-design proof: Route B}\label{sec:route-b}

\subsection{The spectral strategy}

The proof above follows Route A: construct a quantitatively transverse
holomorphic section and use its zero divisor.  We now give a second route
which works directly with points and the Green kernel.  It uses optimal
equal-weight designs for Laplace eigenfunctions and therefore does not use
holomorphic sections or quantitative transversality.

At a high level there are three steps.
\begin{enumerate}
\item \emph{Low-frequency cancellation.}  Choose an equal-weight design
that integrates exactly every Laplace mode up to
$\Lambda\asymp\sqrt N$.  It therefore cancels the low-frequency part of the
mean-zero Green kernel.  This is the spectral analogue of an exact
equal-weight quadrature rule.
\item \emph{High-frequency localization.}  The remaining kernel is bounded
above by a rapidly decaying function of $\Lambda \dist_{\geo}(z,w)$, apart from its favorable
negative logarithmic singularity at $z=w$.
\item \emph{Local counting.}  The design has only $O((R+1)^2)$ points in a
ball of radius $R/\Lambda$.  Summing the localized remainder over successive
$\Lambda^{-1}$-annuli
then gives a convergent series independent of $N$.
\end{enumerate}
The cutoff $\Lambda\asymp\sqrt N$ is where spectral dimension and
spatial point
spacing meet: Weyl's law gives $O(\Lambda^2)$ modes, while localization
occurs at
scale $\Lambda^{-1}\asymp N^{-1/2}$.
The same labeled partition used to build the design gives covering radius
$O(N^{-1/2})$; unlike Route A, this argument does not assert that distinct
spectral nodes are separated.  Sending the mass of each equal-area cell to
its labeled node also gives an explicit transport plan of cost
$O(N^{-1/2})$.

For a configuration $\Wset_N=(w_{1,N},\ldots,w_{N,N})$, write
\[
 V_{\Wset_N}(z):=\sum_{j=1}^N U_{\Xsurf}(z,w_{j,N}).
\]
It is enough to construct $\Wset_N$ such that
\begin{equation}\label{eq:route-b-upper-potential}
 \sup_{z\in \Xsurf}V_{\Wset_N}(z)\leq C
\end{equation}
with $C$ independent of $N$, because then
$\calB_{\beta}(\Wset_N;\Xsurf)\leq \exp(\beta C)$ for every $\beta>0$.

\subsection{Laplace designs with local counting}

Choose the unique metric in the fixed conformal class whose area form is
$\pmeas$.  Positivity and smoothness of $\pmeas$ make this a smooth Riemannian
metric; its distance is $\dist_{\geo}$.  Let $\Delta$ be its nonnegative
Laplace--Beltrami operator.  Choose a real orthonormal eigenbasis
$\Delta\varphi_{\nu}=\lambda_{\nu}^2\varphi_{\nu}$ of $L^2(\Xsurf,\pmeas)$, where
$0=\lambda_0<\lambda_1\leq\cdots$.  A configuration $\Wset_N$ is an
$\Lambda$-design if
\begin{equation}\label{eq:diffusion-design}
 \frac1N\sum_{j=1}^NP(w_{j,N})=\int_{\Xsurf} P\,\diff\pmeas
 \qquad
 \text{for every }P\in
 \operatorname{span}\{\varphi_{\nu}:\lambda_{\nu}\leq \Lambda\}.
\end{equation}
Throughout, $\Lambda\geq1$ denotes the varying spectral cutoff; as fixed in
Remark \ref{rem:conventions}, the letter $L$ is reserved for the degree-one
line bundle of Route A and is never used for a bandwidth.  Every constant
in the estimates below is independent of $\Lambda$ unless explicitly stated
otherwise.

Gariboldi and Gigante~\cite{GariboldiGigante} proved that on every compact
connected oriented $m$-manifold there is an equal-weight $\Lambda$-design
represented by an $N$-tuple whenever
$N\geq C_{\Xsurf}\Lambda^m$.  The precise locators used below refer to their
preprint version~\cite{GariboldiGigantePreprint}, whose Theorem~6 gives this
statement.  Their degree argument is carried out through a continuous map
with values in $\Xsurf^N$ and retains labels on the nodes.  We also need the
following geometric information retained by the same construction.

\begin{remark}
\label{rem:design-source-extraction}
The displacement and labeling information in the next lemma is \emph{not} the
literal statement of the design-existence theorem
\cite[Theorem~6]{GariboldiGigantePreprint}, which asserts only the existence
of an exact equal-weight design at the stated cardinality.  The additional
geometry is retained by the construction used in its proof, and we use the
following four ingredients.  The area-regular partition, with its
inner-ball and outer-ball bounds, is due to Gigante and
Leopardi~\cite{GiganteLeopardi} and is stated in
\cite[Theorem~3]{GariboldiGigantePreprint}.  In the proof of
\cite[Lemma~14]{GariboldiGigantePreprint} we choose the initial point
attached to each cell to be the center of that cell's inner ball.
The paths used there follow a vector field of speed at most one and are
evaluated at a time comparable to $N^{-1/2}$, which is what bounds the
cell-to-node displacement.  Finally, the degree argument in the proof of
\cite[Theorem~6]{GariboldiGigantePreprint} selects one resulting
\emph{labeled} $N$-tuple which is an exact design, so that the design
identity and all the geometric data are obtained by one simultaneous
construction.  As explained in the proof below, the common flow also preserves
distinctness of the initial centers.
\end{remark}

With that source reading understood, the geometric form we need is the
following.

\begin{lemma}
\label{lem:partition-compatible-design}
In dimension two there are fixed constants
$c_{\rm part},C_{\rm part},C_{\rm mov}>0$ with the following property.
Whenever $N\geq C_{\Xsurf}\Lambda^2$, the $\Lambda$-design may be labeled
$\Wset_N=(w_{1,N},\ldots,w_{N,N})$ together with an area-regular partition
$\Xsurf=R_1\cup\cdots\cup R_N$ and points $x_j\in R_j$ such that
\begin{equation}\label{eq:design-partition-geometry}
 \pmeas(R_j)=\frac1N,
 \qquad
 B(x_j,c_{\rm part}N^{-1/2})\subset R_j,
 \qquad
 \operatorname{diam}R_j\leq C_{\rm part}N^{-1/2},
\end{equation}
and
\begin{equation}\label{eq:design-node-displacement}
 \dist_{\geo}(w_{j,N},x_j)\leq C_{\rm mov}N^{-1/2}.
\end{equation}
The labels refer to the original partition cells.
\end{lemma}

\begin{proof}
The area-regular partition and its inner- and outer-ball bounds are
\cite[Theorem~3]{GariboldiGigantePreprint}; the outer-ball inclusion there implies
the diameter estimate in \eqref{eq:design-partition-geometry}.  If a null
boundary remainder is present, assign each of its points to one cell whose
closure contains it.  Enlarging that cell only inside its original closure
preserves its measure and diameter, so the cells cover $\Xsurf$ without
changing the estimates.  In the proof of their Lemma~14, choose the initial
point attached to $R_j$ to be the center $x_j$ of its inner ball. The paths
denoted there by $y_j(P,t)$ follow a vector field of speed at most one, and
the map used in the degree argument evaluates them at time
$12c_2N^{-1/2}$, where $c_2$ is the outer-ball constant from their
Theorem~3; see \cite[proof of Lemma~14]{GariboldiGigantePreprint}.  Hence
\eqref{eq:design-node-displacement}.  The degree argument in the proof of
their Theorem~6 selects one resulting labeled $N$-tuple which is an exact
$\Lambda$-design.  Thus the design and all the stated geometric data can be
chosen simultaneously.  For each fixed $P$, the paths follow the same smooth
vector field
\[
 U_P=\frac{\nabla P}{v_{\varepsilon}(|\nabla P|)},
\]
where $v_{\varepsilon}$ is smooth, positive, and constant near zero.  Its
flow at the common time $12c_2N^{-1/2}$ is a diffeomorphism of the compact
surface, so it maps the distinct inner-ball centers to distinct nodes.
This argument gives no uniform lower bound on their mutual distances.
\end{proof}

The lemma has two immediate consequences.  First, comparing a small ball
with the disjoint inner balls in
\eqref{eq:design-partition-geometry} gives the local counting estimate
\begin{equation}\label{eq:design-local-count}
 \#\bigl(\Wset_N\cap B(x,RN^{-1/2})\bigr)
 \leq C_{\Xsurf}(1+R)^2,
 \qquad x\in \Xsurf,\quad R\geq0.
\end{equation}
Indeed, if a final node belongs to $B(x,RN^{-1/2})$, its initial inner-ball
center lies in $B(x,(R+C)N^{-1/2})$.  The inner balls have disjoint interiors,
radius comparable to $N^{-1/2}$, and area comparable to $N^{-1}$; an area
comparison therefore bounds their number by $C(R+C)^2$.
This argument uses small-ball area comparison when
$(R+C)N^{-1/2}$ is below a fixed radius.  Above that range,
$C(1+R)^2\geq N$ after increasing $C$, so the estimate is automatic.
Second, if $x\in R_j$, then
\begin{equation}
 \dist_{\geo}(x,w_{j,N})\leq\operatorname{diam}R_j+\dist_{\geo}(x_j,w_{j,N})
 \leq(C_{\rm part}+C_{\rm mov})N^{-1/2}.
\label{eq:design-local-distance}
\end{equation}
Consequently the same designs have the natural-scale covering bound
\begin{equation}\label{eq:design-covering-radius}
 \sup_{x\in \Xsurf}\min_{1\leq j\leq N}\dist_{\geo}(x,w_{j,N})
 \leq C_{\Xsurf} N^{-1/2}.
\end{equation}

\begin{corollary}
\label{cor:design-bounded-multiplicity}
For the labeled design above,
\begin{equation}\label{eq:design-bounded-multiplicity}
 \sup_{y\in \Xsurf}\#\{j:w_{j,N}=y\}\leq C_{\Xsurf}.
\end{equation}
Consequently every atom of
$\nu_{\Wset_N}=N^{-1}\sum_{j=1}^N\delta_{w_{j,N}}$ has mass at most $C_{\Xsurf}/N$.
\end{corollary}

\begin{proof}
For any $y\in \Xsurf$, all indices with $w_{j,N}=y$ are counted in
$\Wset_N\cap B(y,N^{-1/2})$.  Taking $R=1$ in
\eqref{eq:design-local-count} proves
\eqref{eq:design-bounded-multiplicity}; division by $N$ gives the assertion
about the empirical measure.
\end{proof}

Transporting the mass of each partition cell to its associated design node gives a coupling with uniformly bounded displacement.  The atomic transport lower bound gives the matching order.

\begin{corollary}
\label{cor:design-direct-transport}
For every $1\leq p\leq\infty$, the partition-compatible designs satisfy
\begin{equation}\label{eq:design-direct-Wasserstein-p}
 c_{\Xsurf} N^{-1/2}\leq W_p(\nu_{\Wset_N},\pmeas)\leq C_{\Xsurf} N^{-1/2}.
\end{equation}
The estimate holds for any fixed Riemannian ground metric on $\Xsurf$.
\end{corollary}

\begin{proof}
We first use the Route B distance $\dist_{\geo}$, whose area form is $\pmeas$.  Any other
fixed Riemannian distance on the compact surface is bi-Lipschitz equivalent
to $d$, and the same comparison constant transfers every $W_p$ estimate,
including $p=\infty$, uniformly in $p$.
Up to their pairwise null intersections, the sets $R_j$ partition $\Xsurf$ and
have mass $1/N$.  Consider the transport plan on $\Xsurf\times \Xsurf$ given by
\[
 \pi_N:=\sum_{j=1}^N
       (\operatorname{id},w_{j,N})_\#(\pmeas|_{R_j}).
\]
Its first marginal is $\pmeas$, and its second marginal is
$N^{-1}\sum_j\delta_{w_{j,N}}=\nu_{\Wset_N}$.  By
\eqref{eq:design-partition-geometry} and
\eqref{eq:design-node-displacement},
\[
 \begin{aligned}
 \int_{\Xsurf\times \Xsurf}\dist_{\geo}(x,y)\,\diff\pi_N(x,y)
 &=\sum_{j=1}^N\int_{R_j}\dist_{\geo}(x,w_{j,N})\,\diff\pmeas(x)\\
 &\leq(C_{\rm part}+C_{\rm mov})N^{-1/2}
       \sum_{j=1}^N\pmeas(R_j)
 \leq C_{\Xsurf} N^{-1/2}.
 \end{aligned}
\]
Every pair in the support of $\pi_N$ has distance at most
$C_{\Xsurf} N^{-1/2}$.  Thus the same plan gives
\[
 \left(\int \dist_{\geo}(x,y)^p\,\diff\pi_N(x,y)\right)^{1/p}
 \leq C_{\Xsurf} N^{-1/2}
\]
for $1\leq p<\infty$, and its essential maximal cost obeys the same bound
when $p=\infty$.  Taking the infimum over plans proves the upper estimate in
\eqref{eq:design-direct-Wasserstein-p}.  The monotonicity
$W_p\geq W_1$ on probability measures, together with the atomic lower bound
\eqref{eq:Wasserstein-lower} proved in Section \ref{sec:consequences}, gives
the lower estimate.  That bound depends only on $\nu$ being supported on
at most $N$ points, so there is no circularity.
\end{proof}

The same coupling controls integration errors for H\"older test functions at the natural spatial scale.

\begin{corollary}
\label{cor:design-holder-discrepancy}
For $0<\alpha\leq1$, measure H\"older discrepancy in the fixed metric $\geo$ by
\[
 [f]_{C^{0,\alpha}}
 :=\sup_{x\neq y}\frac{|f(x)-f(y)|}{\dist_{\geo}(x,y)^\alpha},
 \qquad
 \mathcal D_{\alpha}(\nu,\pmeas)
 :=\sup_{\substack{f\in C^{0,\alpha}(\Xsurf;\mathbb{R})\\
                   [f]_{C^{0,\alpha}}\leq1}}
       \left|\int_{\Xsurf} f\,\diff(\nu-\pmeas)\right|.
\]
Then the partition-compatible designs satisfy
\begin{equation}\label{eq:design-holder-discrepancy}
 c_{\Xsurf} N^{-\alpha/2}
 \leq\mathcal D_{\alpha}(\nu_{\Wset_N},\pmeas)
 \leq C_{\Xsurf} N^{-\alpha/2}.
\end{equation}
More generally, the lower bound holds for every probability measure
supported on at most $N$ points.  Replacing $\geo$ by any other fixed
Riemannian metric changes only the constants.
\end{corollary}

\begin{proof}
For the upper bound, bi-Lipschitz equivalence of the fixed metrics and the
cell-to-node labeling give $\dist_{\geo}(x,w_{j,N})\leq C_{\Xsurf} N^{-1/2}$ for
$x\in R_j$.  Hence
\[
 \begin{aligned}
 \left|\int_{\Xsurf} f\,\diff(\nu_{\Wset_N}-\pmeas)\right|
 &=\left|\sum_{j=1}^N\int_{R_j}
        \bigl(f(w_{j,N})-f(x)\bigr)\,\diff\pmeas(x)\right|\\
 &\leq [f]_{C^{0,\alpha}}
       C_{\Xsurf}^\alpha N^{-\alpha/2}\sum_{j=1}^N\pmeas(R_j).
 \end{aligned}
\]
This proves the desired estimate after absorbing $C_{\Xsurf}^\alpha$ into a
constant uniform for $0<\alpha\leq1$.

For the lower bound, let $S$ be the support of a probability measure $\nu$
with $\#S\leq N$ and write $\psi(x)=\dist_{\geo}(x,S)^\alpha$.  Since
$|a^\alpha-b^\alpha|\leq|a-b|^\alpha$ for $a,b\geq0$, one has
$[\psi]_{C^{0,\alpha}}\leq1$, while $\int_{\Xsurf}\psi\,\diff\nu=0$.
Choose a small-ball constant $C_{\rm vol}$ such that
$\pmeas(B_{\geo}(x,r))\leq C_{\rm vol}r^2$, and enlarge it so that
$C_{\rm vol}\geq1$ and
$r_N=(2C_{\rm vol}N)^{-1/2}$ is always in the permitted small-radius range.
The union of the $r_N$-balls around $S$ has measure at most $1/2$.
Consequently,
\[
 \int_{\Xsurf}\psi\,\diff\pmeas
 \geq\frac12r_N^\alpha
 =\frac12(2C_{\rm vol})^{-\alpha/2}N^{-\alpha/2}
 \geq\frac12(2C_{\rm vol})^{-1/2}N^{-\alpha/2}.
\]
Testing the discrepancy with $\psi$ proves the lower bound.
The final metric assertion follows from bi-Lipschitz equivalence on the
compact surface, which compares the two $C^{0,\alpha}$ seminorms by fixed
powers of the Lipschitz constants.  Those powers are uniformly bounded for
$0<\alpha\leq1$.  Thus both constants in
\eqref{eq:design-holder-discrepancy} may be chosen independently of
$\alpha$ in this interval.
\end{proof}

\begin{remark}
Because $\dist_{\geo}^\alpha$ is again a metric for $0<\alpha\leq1$,
$\mathcal D_{\alpha}$ is exactly the Kantorovich--Rubinstein $W_1$ distance
computed with the snowflaked ground metric $\dist_{\geo}^\alpha$.  Thus
\eqref{eq:design-holder-discrepancy} is an optimal quantization statement,
not only a one-sided quadrature estimate.
\end{remark}

Choose $C_{\rm des}$ so that
\Cref{lem:partition-compatible-design} applies whenever
$N\geq C_{\rm des}\Lambda^2$.  Next choose a fixed
$a>0$ so small that $4C_{\rm des}a^2\leq1$, write $\Lambda=a\sqrt N$, and apply
their theorem with bandwidth $2\Lambda$.  The numerical choice ensures
$N\geq C_{\rm des}(2\Lambda)^2$.  Then both
\eqref{eq:diffusion-design} and \eqref{eq:design-local-count} hold, with
$N^{-1/2}$ comparable to $\Lambda^{-1}$.  The bandwidth is $2\Lambda$ rather
than $\Lambda$ because the spectral cutoff
introduced in \eqref{eq:green-low-frequency} below is supported in
$\lambda_{\nu}\leq2\Lambda$, so cancellation must be available up to that
frequency.

\subsection{The high-frequency Green remainder}

Fix $\chi\in C_c^\infty([0,\infty))$ with
$\chi=1$ on $[0,1]$ and $\chi=0$ on $[2,\infty)$.  Because the Riemannian
area form is $\pmeas$ and $\Delta$ is nonnegative, our convention for $\diff^c$
gives
\[
 \ddc f=-\frac1{2\pi}(\Delta f)\,\pmeas.
\]
Consequently, \eqref{eq:green-normalization} fixes the spectral coefficient:
\[
 U_{\Xsurf}(z,w)=-2\pi\sum_{\nu\geq1}
 \frac{\varphi_{\nu}(z)\varphi_{\nu}(w)}{\lambda_{\nu}^2}
\]
in the distributional sense.  Its smooth low-frequency part is
\begin{equation}\label{eq:green-low-frequency}
 U_{\leq \Lambda}(z,w):=-2\pi\sum_{\nu\geq1}
 \chi(\lambda_{\nu}/\Lambda)
 \frac{\varphi_{\nu}(z)\varphi_{\nu}(w)}{\lambda_{\nu}^2}
\end{equation}
and the complementary remainder is
$R_{\Lambda}:=U_{\Xsurf}-U_{\leq \Lambda}$.

We first record the uniform form of smooth spectral localization needed below.  Its emphasis is on the dependence of the constants on the multiplier, which is important when summing dyadic frequency bands.

\begin{lemma}\label{lem:uniform-spectral-localization}
Let $I$ be a compact subinterval of $(0,\infty)$, and let $\mathcal H$ be a family of smooth functions supported in $I$ such that
\[
 \sup_{h\in\mathcal H}\|h^{(j)}\|_{L^\infty}<\infty
 \qquad\text{for every integer }j\geq0.
\]
For every $M>0$ there are an integer $J_M\geq0$ and a constant $C_{M,I}$, depending only on $M$, $I$, and the fixed geometry, such that the kernel of $h(\sqrt\Delta/T)$ satisfies
\[
 |K_{h,T}(z,w)|\leq C_{M,I} T^2
 (1+T\dist_{\geo}(z,w))^{-M}
 \sum_{j=0}^{J_M}\|h^{(j)}\|_{L^\infty},
 \qquad h\in\mathcal H,\quad T\geq1,\quad z,w\in\Xsurf.
\]
In particular, the right-hand side is bounded by $C_M T^2(1+T\dist_{\geo}(z,w))^{-M}$ uniformly over $h\in\mathcal H$.
\end{lemma}

\begin{proof}
We apply the smooth functional calculus of Kerkyacharian and Petrushev
\cite[Theorem~3.1]{KerkyacharianPetrushev}.  Fix an integer
$k>\max\{M,d+1\}$, where $d$ is a doubling exponent for the metric measure
space $(\Xsurf,\dist_{\geo},\pmeas)$.  Their theorem, applied with
$\delta=T^{-1}$ and $f=h$, gives
\[
 |K_{h,T}(z,w)|\leq c_k(h)
 \frac{(1+T\dist_{\geo}(z,w))^{-k}}
 {\bigl[\pmeas(B(z,T^{-1}))\pmeas(B(w,T^{-1}))\bigr]^{1/2}}.
\]
Their explicit constant estimate, equation~(3.3), controls $c_k(h)$ by a
constant depending on $k$, the fixed support interval $I$, and the geometry,
multiplied by $\|h\|_\infty+\|h^{(k)}\|_\infty$.
The odd-derivative conditions at the origin hold because $I\Subset(0,\infty)$.
On the compact Riemannian surface, uniformly for $T\geq1$,
$\pmeas(B(z,T^{-1}))\asymp T^{-2}$.  The asserted estimate follows with
$J_M=k$, since $(1+u)^{-k}\leq(1+u)^{-M}$.
\end{proof}

After cancellation of the low frequencies, it remains to bound the Green kernel remainder from above.  The logarithmic singularity makes a one-sided estimate the appropriate formulation.

\begin{lemma}\label{lem:green-high-frequency}
For every $M>0$ there is $C_M$ such that, for $\Lambda\geq1$ and $z\neq w$,
\begin{equation}\label{eq:green-high-frequency-bound}
 R_{\Lambda}(z,w)\leq C_M(1+\Lambda \dist_{\geo}(z,w))^{-M}.
\end{equation}
\end{lemma}

\begin{proof}
When $\Lambda \dist_{\geo}(z,w)\geq1$, decompose the multiplier
$(1-\chi(\lambda/\Lambda))\lambda^{-2}$ into smooth dyadic annuli
$\lambda\asymp2^k\Lambda$, $k\geq0$, using a smooth partition of unity on the
high-frequency half-line.  By Lemma~\ref{lem:uniform-spectral-localization}, the kernel of a normalized multiplier $h(\sqrt\Delta/T)$, with $T\geq1$ and $h$ in a uniformly bounded smooth family supported in $(1/2,2)$, satisfies
\[
 |K_{h,T}(z,w)|\leq C_M T^2(1+T\dist_{\geo}(z,w))^{-M}.
\]
Applying it to the normalized multiplier on
the $k$-th annulus gives
\[
 C_M(1+2^k\Lambda \dist_{\geo}(z,w))^{-M}.
\]
Indeed, the kernel itself has size $(2^k\Lambda)^2$, while the factor
$\lambda^{-2}$ contributes $(2^k\Lambda)^{-2}$.  Since $\Lambda \dist_{\geo}(z,w)\geq1$,
\[
 \sum_{k\geq0}(1+2^k\Lambda \dist_{\geo}(z,w))^{-M}
 \leq C_M(1+\Lambda \dist_{\geo}(z,w))^{-M}.
\]
Summing in $k$ proves
\eqref{eq:green-high-frequency-bound} off the $\Lambda^{-1}$-neighborhood of the
diagonal.

The dyadic partition is fixed once and for all.  After rescaling the
$k$-th annulus to $(1/2,2)$, the normalized multiplier belongs to a bounded
set in every $C^s$ seminorm, uniformly in $k$ and $\Lambda$.  Hence the
constants in the multiplier estimate are uniform over all annuli, as
required by the preceding sum.

It remains to treat $0<\dist_{\geo}(z,w)\leq \Lambda^{-1}$.  Let $\Pi_0$ denote orthogonal
projection onto the constants.  At the operator level the definition of
$R_{\Lambda}$ is the exact identity
\begin{equation}\label{eq:green-remainder-operator}
 R_{\Lambda}=-2\pi\bigl(1-\chi(\sqrt{\Delta}/\Lambda)\bigr)
 \Delta^{-1}(I-\Pi_0).
\end{equation}
For an upper bound near the diagonal, a heat-kernel comparison avoids any
local oscillatory-integral expansion.  Write
\[
 U_{\mathrm{heat},\Lambda}:=\e^{-\Delta/\Lambda^2}U_{\Xsurf}.
\]
Since
\[
 (I-\e^{-\Delta/\Lambda^2})\Delta^{-1}(I-\Pi_0)
 =\int_0^{\Lambda^{-2}}\e^{-t\Delta}(I-\Pi_0)\,dt,
\]
the kernel of $U_{\Xsurf}-U_{\mathrm{heat},\Lambda}$ is
\begin{equation}\label{eq:green-heat-remainder}
 -2\pi\int_0^{\Lambda^{-2}}\bigl(H_t(z,w)-1\bigr)\,dt,
\end{equation}
where $H_t$ is the heat kernel relative to $\pmeas$.  Here the kernel of
$\Pi_0$ is $1$ because $\pmeas(\Xsurf)=1$.  Positivity of $H_t$ therefore
gives the pointwise upper bound
\begin{equation}\label{eq:green-heat-remainder-upper}
 U_{\Xsurf}(z,w)-U_{\mathrm{heat},\Lambda}(z,w)\leq 2\pi \Lambda^{-2}.
\end{equation}

It remains only to compare the heat cutoff with the compact spectral cutoff.
Extend $\chi$ evenly to $\R$ and consider
\[
 m(r):=
 \begin{cases}
  \dfrac{\e^{-r^2}-\chi(r)}{r^2},&r\neq0,\\[6pt]
  -1,&r=0.
 \end{cases}
\]
Because $\chi=1$ near the origin, $m$ is an even Schwartz function.  Spectral
calculus gives
\begin{equation}\label{eq:green-cutoff-comparison}
 U_{\mathrm{heat},\Lambda}-U_{\leq \Lambda}
 =-2\pi \Lambda^{-2}m(\sqrt\Delta/\Lambda)(I-\Pi_0).
\end{equation}
For completeness, its kernel is uniformly bounded without invoking a
parameter-dependent Green parametrix.  Since
$|m(r)|\leq C(1+r^2)^{-2}$, Cauchy--Schwarz reduces the claim to
\[
 \sup_{z\in \Xsurf}\sum_{\nu\geq0}
 (1+\lambda_{\nu}^2/\Lambda^2)^{-2}|\varphi_{\nu}(z)|^2\leq C\Lambda^2.
\]
Using
\[
 (1+x)^{-2}=\int_0^\infty t \e^{-t}\e^{-tx}\,dt
\]
and the standard diagonal heat-kernel estimate
$H_s(z,z)\leq C(1+s^{-1})$ on a compact surface
\cite[Theorem~1.1]{Ludewig}, with the large-time bound following from the
semigroup property, the left-hand side is at
most
\[
 C\int_0^\infty t \e^{-t}
 \bigl(1+\Lambda^2/t\bigr)\,dt\leq C\Lambda^2.
\]
Thus the kernel in \eqref{eq:green-cutoff-comparison} has absolute value at
most a constant independent of $\Lambda,z,w$.  Combining this fact with
\eqref{eq:green-heat-remainder-upper} gives
\begin{equation}\label{eq:green-remainder-near-diagonal}
 \sup_{0<\dist_{\geo}(z,w)\leq \Lambda^{-1}}R_{\Lambda}(z,w)\leq C.
\end{equation}
Since $1+\Lambda \dist_{\geo}(z,w)\leq2$ in this region, enlarging $C_M$ by a factor
depending only on $M$ turns this estimate into
\eqref{eq:green-high-frequency-bound} near the diagonal as well.
This completes the proof for every $\Lambda\geq1$.
\end{proof}

\subsection{Completion of Route B}

Because $\Wset_N$ is a $2\Lambda$-design and every nonconstant
eigenfunction has mean zero, \eqref{eq:green-low-frequency} gives
\begin{equation}\label{eq:low-frequency-cancellation}
 \sum_{j=1}^NU_{\leq \Lambda}(z,w_{j,N})=0.
\end{equation}
Indeed, for every index in the support of the cutoff,
\[
\sum_{j=1}^N\varphi_{\nu}(w_{j,N})
=N\int_{\Xsurf}\varphi_{\nu}\,\diff\pmeas=0,
\]
and the low-frequency sum is finite, so cancellation may be performed term
by term.
Consequently
\[
 V_{\Wset_N}(z)=\sum_{j=1}^NR_{\Lambda}(z,w_{j,N}).
\]
If $z$ is one of the nodes, the left-hand side is $-\infty$ and the desired
upper bound is automatic.  Otherwise divide $\Xsurf$ into the annuli
\[
 m\leq \Lambda \dist_{\geo}(z,w)<m+1,\qquad m=0,1,2,\ldots.
\]
By \eqref{eq:design-local-count}, the number of nodes in the $m$-th annulus
is at most the number in the ball of radius $(m+1)\Lambda^{-1}$, hence at most
$C(m+2)^2$.  Indeed, since $\Lambda=a\sqrt N$, this ball has radius
$R N^{-1/2}$ with $R=(m+1)/a$, and
\eqref{eq:design-local-count} gives
$C(1+(m+1)/a)^2\leq C_a(m+2)^2$.  Taking $M>3$ in Lemma
\ref{lem:green-high-frequency} therefore yields
\begin{equation}\label{eq:route-b-final-sum}
 V_{\Wset_N}(z)
 \leq C_M\sum_{m=0}^\infty (m+2)^2(1+m)^{-M}\leq C.
\end{equation}
The series converges precisely because its summand is
$O(m^{2-M})$ and $M>3$.
This is \eqref{eq:route-b-upper-potential}, uniformly in $N$ and $z$.
The spectral design and the one-sided remainder bound thus provide a second construction of uniformly packed configurations.

\begin{theorem}\label{thm:route-b}
There is a constant $a=a(\Xsurf,\pmeas)>0$ such that, for every sufficiently large
$N$, there is a configuration $\Wset_N\in \Xsurf^N$
which is an exact equal-weight quadrature rule for all Laplace frequencies
$\lambda_{\nu}\leq2a\sqrt N$, satisfies the local counting estimate
\eqref{eq:design-local-count}, the natural-scale covering bound
\eqref{eq:design-covering-radius}, the bounded-multiplicity estimate
\eqref{eq:design-bounded-multiplicity}, and the all-order transport bounds
\eqref{eq:design-direct-Wasserstein-p}, as well as the sharp H\"older
quadrature estimate \eqref{eq:design-holder-discrepancy}.  It also obeys
\eqref{eq:route-b-upper-potential}.  Hence
$\calB_{\beta}(\Wset_N;\Xsurf)\leq C_{\Xsurf,\pmeas,\beta}$ for every $\beta>0$.
This is exactly the uniform upper bound asserted in Theorem \ref{thm:main};
the finitely many smaller values of $N$ are handled as in Section
\ref{sec:main-proof}.
This proof does not use holomorphic quantitative transversality.
\end{theorem}

For comparison, the partition geometry also suggests configurations obtained from balanced cells.

\begin{remark}
The cell-moment conditions suggested by the elementary quadrature approach
are sufficient but not necessary.  The spectral design cancels all modes
up to the natural frequency $\Lambda\asymp\sqrt N$ at once; the localized
remainder
estimate controls all smaller scales.  Thus it supplies the kind of
balanced Green quadrature estimate one would obtain from individually
isotropic cells, without having to construct such cells.
\end{remark}

\section{Immediate consequences and comparison of the two
constructions}
\label{sec:consequences}

\subsection{Quantitative equidistribution}

Both routes are now complete, and we record the distributional consequence that
follows from either.  Its mechanism is independent of the constructions, using
only the uniform exponential bound.  The systematic
development of this observation --- sharp quantization for every concave
modulus of continuity, mesoscopic counting, negative Sobolev estimates, and the
variational theory of minimizing configurations --- is carried out in the
companion paper \cite{PaperB}, whose hypotheses are supplied by Theorem 
\ref{thm:main}.

We first state the consequence for the zero divisors obtained from Route A.

\begin{theorem}\label{thm:equidistribution}
For the divisors $D_N$ above, the probability measures
\[
 \nu_N:=\frac1N[D_N]=\frac1N\sum_{j=1}^{N}\delta_{w_{j,N}}
\]
obey
\begin{equation}\label{eq:quantitative-equidistribution}
 \left|\int_{\Xsurf}\phi\,\diff(\nu_N-\pmeas)\right|
 \leq\frac{C}{N}\|\ddc\phi\|_{L^\infty(\Xsurf)}
 \qquad(\phi\in C^\infty(\Xsurf)).
\end{equation}
In particular, $\nu_N\to\pmeas$ weakly.
\end{theorem}

\begin{proof}
The potential identity turns the logarithmic integrability estimate
into an $L^1$ bound for the Green potential, and $\ddc$ of that potential
is exactly $N$ times the signed measure we wish to control.  Integration
by parts does the rest.

Write
\[
 V_N(z):=\sum_{j=1}^{N}U_{\Xsurf}(z,w_{j,N}).
\]
By the potential identity and Theorem \ref{thm:log-L1},
\[
 \|V_N\|_{L^1(\Xsurf,\pmeas)}
 \leq \|\log|s_N|\|_{L^1}+|m_N|\leq 2C.
\]
Moreover,
\[
 \ddc V_N=[D_N]-N\pmeas.
\]
Let $\phi\in C^\infty(\Xsurf)$.  Since $V_N$ is locally integrable and, by the
Poincar\'e--Lelong identity \eqref{eq:PL},
$\ddc V_N=[D_N]-N\pmeas$ is a signed measure, distributional integration by
parts yields
\[
 \int_{\Xsurf}\phi\,\diff(\nu_N-\pmeas)
 =\frac1N\langle\ddc V_N,\phi\rangle
 =\frac1N\langle V_N,\ddc\phi\rangle.
\]
The $L^1$ bound for $V_N$ proves
\eqref{eq:quantitative-equidistribution}.  Its right-hand side tends to
zero, and smooth test functions determine weak convergence of probability
measures on the compact space $\Xsurf$.
\end{proof}

\subsection{Discrepancy and transport estimates}

The two remaining ingredients are a discrepancy estimate and a matching
transport lower bound.  Both are special cases of a general theory of
configurations with bounded $\beta$-exponential Green energy, developed
systematically in \cite{PaperB}; we state here only what the present
paper uses, in the form in which it is used, and refer to \cite{PaperB}
for the sharp and more general statements.

The next statement records the energy-to-discrepancy and transport implications used here, in the form supplied by the companion theory.

\begin{proposition}
\label{prop:energy-consequences}
Let $\Wset_N\in \Xsurf^N$ and suppose
$\calB_{\beta}(\Wset_N;\Xsurf)\leq K<\infty$
for one fixed $\beta>0$ and every $N$.  Then $K\geq1$,
\begin{equation}\label{eq:energy-L1-control}
 \bigl\|V_{\Wset_N}\bigr\|_{L^1(\Xsurf,\pmeas)}\leq\frac{2(1+\log K)}{\beta},
\end{equation}
and consequently
\begin{equation}\label{eq:energy-discrepancy}
 \left|\int_{\Xsurf}\phi\,\diff(\nu_{\Wset_N}-\pmeas)\right|
 \leq\frac{2(1+\log K)}{\beta N}\|\ddc\phi\|_{\infty},
 \qquad\phi\in C^\infty(\Xsurf).
\end{equation}
Moreover
\begin{equation}\label{eq:Wasserstein-bound}
 W_1(\nu_{\Wset_N},\pmeas)\leq C_{\Xsurf,K,\beta}N^{-1/2},
\end{equation}
and every probability measure $\nu$ on $\Xsurf$ supported on at most $N$ points
satisfies
\begin{equation}\label{eq:Wasserstein-lower}
 W_1(\nu,\pmeas)\geq c_{\Xsurf} N^{-1/2},
\end{equation}
so that the exponent in \eqref{eq:Wasserstein-bound} is optimal.
\end{proposition}

\begin{proof}
Each Green kernel has mean zero, so $\int_{\Xsurf} V_{\Wset_N}\,\diff\pmeas=0$,
and Jensen's inequality gives $K\geq1$.  Markov's inequality centered at
$a_K:=(\log K)/\beta$ gives the exponential tail
$\pmeas\{V_{\Wset_N}>a_K+t\}\leq \e^{-\beta t}$ for $t\geq0$; integrating it by the
layer-cake formula bounds $\int_{\Xsurf}(V_{\Wset_N})_+\,\diff\pmeas$ by
$(1+\log K)/\beta$, and the
mean-zero identity makes the integrals of the positive and negative parts equal.
This is \eqref{eq:energy-L1-control}.  Since
$\ddc V_{\Wset_N}=N(\nu_{\Wset_N}-\pmeas)$ as distributions, and $V_{\Wset_N}$
is locally integrable, distributional integration by parts turns
\eqref{eq:energy-L1-control} into \eqref{eq:energy-discrepancy}.

For \eqref{eq:Wasserstein-bound}, choose a function $\phi$ that is one-Lipschitz
for the conformal metric $\geo_{\pmeas}$ whose area form is $\pmeas$, and write
$\phi_t=\e^{-t\Delta_{\pmeas}}\phi$.
On the fixed compact Riemannian surface, the short-time heat-kernel
parametrix and its differentiated remainder estimates give, for
$0<t\leq1$ and $a=0,2$,
\[
 |\nabla_x^a H_t(x,y)|
 \leq C_a t^{-(2+a)/2}
 \e^{-\dist_{\geo_{\pmeas}}(x,y)^2/(c t)}.
\]
Here $c>0$ and the constants depend only on the fixed metric.  The
local parametrix supplies the estimate near the diagonal, while the
off-diagonal heat-kernel estimates handle its complement.  The differentiated estimate also follows from
\cite[Corollary~1.2]{HsuHeatDerivatives}, combined with the standard
Gaussian upper bound for $H_t$ on a compact Riemannian manifold.
Indeed, Hsu's estimate bounds $|\nabla_x^2H_t(x,y)|$ by a constant
times $(\dist_{\geo_{\pmeas}}(x,y)/t+t^{-1/2})^2H_t(x,y)$.
The polynomial factor is absorbed into a Gaussian with a larger
constant in the denominator.  Since $H_t$ preserves constants,
\[
 \int_{\Xsurf}\nabla_x^2 H_t(x,y)\,\diff\pmeas(y)=0.
\]
For a one-Lipschitz function $\phi$, subtracting $\phi(x)$ inside the
heat-kernel integrals therefore yields
\begin{align*}
 |\phi_t(x)-\phi(x)|
 &\leq \int_{\Xsurf}H_t(x,y)\dist_{\geo_{\pmeas}}(x,y)
 \,\diff\pmeas(y)\leq C\sqrt t,\\
 |\nabla^2\phi_t(x)|
 &\leq \int_{\Xsurf}|\nabla_x^2H_t(x,y)|
 \dist_{\geo_{\pmeas}}(x,y)\,\diff\pmeas(y)
 \leq Ct^{-1/2}.
\end{align*}
Using $\ddc f=-(2\pi)^{-1}(\Delta_{\pmeas}f)\pmeas$, we obtain
\begin{equation}\label{eq:heat-smoothing-Lipschitz}
 \|\phi-\phi_t\|_{L^\infty}\leq C\sqrt t,
 \qquad
 \|\ddc\phi_t\|_{L^\infty}\leq Ct^{-1/2}.
\end{equation}
Applying \eqref{eq:energy-discrepancy} to $\phi_t$ and using that both measures
have mass one,
\[
 \left|\int_{\Xsurf}\phi\,\diff(\nu_{\Wset_N}-\pmeas)\right|
 \leq2\|\phi-\phi_t\|_{\infty}+\frac{C_{K,\beta}}{N}\|\ddc\phi_t\|_{\infty}
 \leq C\left(\sqrt t+\frac1{N\sqrt t}\right).
\]
Choosing $t=N^{-1}$ and taking the supremum over such $\phi$ gives
\eqref{eq:Wasserstein-bound} for the distance of $\geo_{\pmeas}$,
hence for any fixed Riemannian distance, all of which are
bi-Lipschitz equivalent on the compact surface.

For \eqref{eq:Wasserstein-lower}, write $S=\operatorname{supp}\nu$, where
$\#S\leq N$, and use the one-Lipschitz function
$\psi(x)=\dist_{\geo}(x,S)$, which vanishes on
$S$.  Fix $C_{\rm vol}$ and $r_*>0$ with $\pmeas(B_{\geo}(x,r))\leq
C_{\rm vol}r^2$ for $0<r\leq r_*$.  Increase $C_{\rm vol}$ so that
$(2C_{\rm vol})^{-1/2}\leq r_*$, and set
$r_N=(2C_{\rm vol}N)^{-1/2}$.  The union of the
$r_N$-balls around $S$ then has measure at most $1/2$, so
$\int_{\Xsurf}\psi\,\diff\pmeas\geq r_N/2$, while $\int_{\Xsurf}\psi\,\diff\nu=0$.
Testing the Kantorovich-Rubinstein duality with $\psi$ gives
$W_1(\nu,\pmeas)\geq r_N/2=c_{\Xsurf} N^{-1/2}$.
\end{proof}

\begin{proof}[Proof of Corollary \ref{cor:intro-equidistribution}]
The discrepancy estimate is \eqref{eq:energy-discrepancy} and the transport
upper bound is \eqref{eq:Wasserstein-bound}; the hypothesis
$\sup_N\calB_{\beta}(\Wset_N;\Xsurf)\leq K$ holds for the holomorphic divisors by
Theorem \ref{thm:main} and for the Route B configurations by Theorem
\ref{thm:route-b}.  The lower bound is \eqref{eq:Wasserstein-lower},
since each $\nu_{\Wset_N}$ is supported on at most $N$ points.
\end{proof}

The two constructions yield the same packing conclusion but provide different additional geometric information.

\begin{remark}
\label{rem:route-comparison}
The two proofs are not variants of one argument, and their by-products differ.

Route A produces a holomorphic section, hence a divisor constrained to lie in a
single Abel fiber over $[L^N]$, and it yields the two-sided local norm model
\eqref{eq:near-zero-model}.  That model gives genuine microscopic geometry:
simple zeros, separation at the natural scale, bounded covering radius, and the
uniform Voronoi geometry of Corollary \ref{cor:transverse-voronoi}.  It is
also what makes the divisors amenable to the condition theory of \cite{PaperC}.

Route B imposes no Picard class and produces no holomorphic object.  In
exchange, the cell-to-node coupling retains transport information at the
microscopic scale, giving the sharp order $N^{-1/2}$ simultaneously for every
$1\leq p\leq\infty$ in \eqref{eq:design-direct-Wasserstein-p}, which the
energy hypothesis alone does not give for $p>1$.  The common-flow argument in
\Cref{lem:partition-compatible-design} gives distinct nodes, but does not give
separation at the natural scale.

Neither construction is subsumed by the other.  Route A also implies
\eqref{eq:route-b-upper-potential}, since the potential identity writes
$V_{\Wset_N}=\log|s_N|_{\heo^N}-m_N$, where the logarithmic norm is uniformly
bounded above and $m_N$ is uniformly bounded in absolute value; but that
implication does not explain the purely potential-theoretic cancellation
furnished by spectral designs, which requires no holomorphic structure and no
transversality theorem.
\end{remark}

The polynomial analogy motivates the problem, while the intrinsic Green formulation applies on every compact Riemann surface.

\begin{remark}\label{rem:bombieri-scope}
On $\Xsurf=\bP^1$ the bundle $\cO(D)$ is isomorphic to $\cO(N)$, the canonical
section becomes a scalar multiple of $\prod_j(z-w_j)$ in an affine chart, and
with the Fubini--Study metric its logarithmic norm is the sum of spherical
Green kernels plus a constant.  Theorem \ref{thm:main} is therefore a
coordinate-free higher-genus analogue of controlling the normalized
$L^\beta$ norm of a polynomial by arranging its zeros.  For $\mathfrak g>1$
the correct global object
is the section of a line bundle rather than a product of scalar linear factors,
and no theta-functional or prime-form representation is needed.  We
prove only the uniform packing statement \Cref{thm:main}.  We do not claim a
sharp multiplicative inequality for arbitrary pairs of sections, and we do not
determine the optimal constant or classify minimizing configurations.  The
variational questions concerning the optimal energy and its minimizing
configurations are taken up in \cite{PaperB}.
\end{remark}

\section{Sharpness of the packing constant}\label{sec:sharpness}

\subsection{A positive lower bound}

Theorem \ref{thm:main} asserts that the pointwise supremum of the Green
potential can be kept bounded, and Corollary \ref{cor:intro-equidistribution}
shows that this forces optimal-order equidistribution.  It is natural to ask
how small the constant $C_{\Xsurf,\pmeas}$ can be.  The answer is that it cannot be
taken arbitrarily small, and indeed cannot tend to zero as $N$ grows: the
packing problem has a genuine positive floor, determined by the surface
alone.  We prove this here, since it shows that Theorem \ref{thm:main} is
sharp in form and not merely in order, and since the argument is short
given the tools already assembled.

The preceding upper packing bound has a complementary lower bound.  Atomic transport prevents the Green potential of an arbitrary configuration from becoming uniformly small.

\begin{theorem}
\label{thm:packing-floor}
There is a constant $a_{\Xsurf}>0$, depending only on $(\Xsurf,\pmeas,\geo)$,
such that every $N\geq1$ and every configuration $\Wset\in \Xsurf^N$ satisfy
\begin{equation}
\label{eq:packing-floor}
\sup_{z\in \Xsurf}V_{\Wset}(z)\geq\frac{a_{\Xsurf}}{2},
\qquad
\bigl\|V_{\Wset}\bigr\|_{L^1(\Xsurf,\pmeas)}\geq a_{\Xsurf} .
\end{equation}
Consequently, for every $\beta>0$,
\begin{equation}\label{eq:theta-floor}
 \Theta_{N,\beta}(\Xsurf)-1\geq\frac{\beta^2a_{\Xsurf}^2}{8},
\end{equation}
so no constant $C_{\Xsurf,\pmeas}$ in \eqref{eq:main-pointwise-packing}
can be smaller than $a_{\Xsurf}/2$, and $\Theta_{N,\beta}(\Xsurf)$ cannot
converge to the Jensen value $1$ as $N\to\infty$.
\end{theorem}

\begin{proof}
There are two steps.  First we show $A_{\Wset}:=\|V_{\Wset}\|_{L^1}$ cannot be
small, by playing the heat-smoothed upper bound for $W_1$ against the atomic
lower bound; optimizing the smoothing time produces the constant.  Then
the mean-zero identity and convexity of $\e^x-1-x$ convert that into the
stated bounds.

Write $A_{\Wset}:=\|V_{\Wset}\|_{L^1(\Xsurf,\pmeas)}$ and $\nu_{\Wset}$ for the
empirical measure.
First, $A_{\Wset}>0$: otherwise $V_{\Wset}=0$ as a distribution, contradicting
$\ddc V_{\Wset}=N(\nu_{\Wset}-\pmeas)$ because $\nu_{\Wset}$ is atomic and
$\pmeas$ is smooth. We now bound $A_{\Wset}$ below.  Since
$\ddc V_{\Wset}=N(\nu_{\Wset}-\pmeas)$, the argument proving
\Cref{prop:energy-consequences} applies verbatim, with the $L^1$ norm in
place of the energy bound.  Testing a one-Lipschitz $\phi$ against its
heat regularization $\phi_t$ and using \eqref{eq:heat-smoothing-Lipschitz}
gives, for $0<t\leq1$,
\[
 \left|\int_{\Xsurf}\phi\,\diff(\nu_{\Wset}-\pmeas)\right|
 \leq2\|\phi-\phi_t\|_{\infty}+\frac1N\|V_{\Wset}\|_{L^1}\|\ddc\phi_t\|_{\infty}
 \leq C_{\Xsurf}\left(\sqrt t+\frac{A_{\Wset}}{N\sqrt t}\right),
\]
whence, by Kantorovich--Rubinstein duality,
\begin{equation}\label{eq:floor-transport-upper}
W_1(\nu_{\Wset},\pmeas)\leq C_{\Xsurf}
\left(\sqrt t+\frac{A_{\Wset}}{N\sqrt t}\right),\qquad 0<t\leq1 .
\end{equation}
On the other hand $\nu_{\Wset}$ is supported on at most $N$ points, so
\eqref{eq:Wasserstein-lower} gives the matching lower bound
\begin{equation}\label{eq:floor-transport-lower}
 W_1(\nu_{\Wset},\pmeas)\geq c_{\Xsurf} N^{-1/2}.
\end{equation}
Suppose first that $A_{\Wset}\leq N$.  Choosing $t=A_{\Wset}/N\in(0,1]$ in
\eqref{eq:floor-transport-upper} and comparing with
\eqref{eq:floor-transport-lower} yields
\[
 c_{\Xsurf} N^{-1/2}\leq2C_{\Xsurf}\sqrt{\frac{A_{\Wset}}{N}},
\]
that is, $A_{\Wset}\geq(c_{\Xsurf}/2C_{\Xsurf})^2=:a_{\Xsurf}>0$.  If instead
$A_{\Wset}>N\geq1$ the same conclusion holds after decreasing $a_{\Xsurf}$ below
$1$.  This proves the second inequality in \eqref{eq:packing-floor}.

For the first, recall that $\int_{\Xsurf} V_{\Wset}\,\diff\pmeas=0$, so the
integrals of the positive and negative parts of $V_{\Wset}$ agree and
\begin{equation}\label{eq:floor-positive-part}
 \int_{\Xsurf}(V_{\Wset})_+\,\diff\pmeas=\frac{A_{\Wset}}{2}\geq\frac{a_{\Xsurf}}{2}.
\end{equation}
Since $\pmeas$ is a probability measure,
$\sup_{\Xsurf} V_{\Wset}\geq\int_{\Xsurf}(V_{\Wset})_+\,\diff\pmeas$, which
is the first inequality in \eqref{eq:packing-floor}.

Finally, $\e^x-1-x\geq x^2/2$ for $x\geq0$ and $\e^x-1-x\geq0$ for all real $x$,
so, using $\int_{\Xsurf}V_{\Wset}\,\diff\pmeas=0$ once more and then Cauchy--Schwarz,
\[
 \calB_{\beta}(\Wset;\Xsurf)-1
 =\int_{\Xsurf}\bigl(\exp(\beta V_{\Wset})-1-\beta V_{\Wset}\bigr)\diff\pmeas
 \geq\frac{\beta^2}{2}\int_{\Xsurf}(V_{\Wset})_+^2\,\diff\pmeas
 \geq\frac{\beta^2}{2}\left(\int_{\Xsurf}(V_{\Wset})_+\,\diff\pmeas\right)^2
\]
where the second step is Cauchy--Schwarz against the probability measure
$\pmeas$. By \eqref{eq:floor-positive-part} the right-hand side is at least
$\beta^2a_{\Xsurf}^2/8$, and taking the infimum over $\Wset\in \Xsurf^N$ gives
\eqref{eq:theta-floor}.
\end{proof}

The lower bound just proved and the uniform second-moment estimate from Route A identify the quadratic dependence on the exponent for small positive $\beta$.

\begin{corollary}\label{cor:quadratic-bergman}
There are constants $0<c_{\Xsurf}\leq C_{\Xsurf}<\infty$, depending only on the prescribed geometry, such that
\[
 c_{\Xsurf}\beta^2\leq\Theta_{N,\beta}(\Xsurf)-1
 \leq C_{\Xsurf}\beta^2
\]
for every $N\geq1$ and $0<\beta\leq1$.
\end{corollary}

\begin{proof}
The lower bound follows from \eqref{eq:theta-floor}. For the configurations constructed by Route A, the potential identity and \Cref{thm:log-L1} give a uniform bound
$\int_{\Xsurf}V_{\Wset_N}^2\,\diff\pmeas\leq C_2$.
For each of the finitely many smaller degrees, fix a configuration of distinct points. Its Green potential is bounded above and has finite second moment, since its singularities are logarithmic. Enlarging the constants if necessary, we therefore have the same uniform second-moment bound for all $N\geq1$. We also have
$\sup_{\Xsurf}V_{\Wset_N}\leq C_0$ and
$\int_{\Xsurf}V_{\Wset_N}\,\diff\pmeas=0$.
For real $v$ and $0<\beta\leq1$,
\[
0\leq \exp(\beta v)-1-\beta v
=\beta^2v^2\int_0^1(1-t)\exp(t\beta v)\,\diff t
\leq \tfrac12\beta^2v^2\exp(v_+).
\]
Integrating with $v=V_{\Wset_N}$ gives
$\Theta_{N,\beta}(\Xsurf)-1\leq\tfrac12\exp((C_0)_+)C_2\beta^2$.
\end{proof}

The preceding bounds leave several extremal questions open.

\begin{remark}\label{rem:floor-scope}
Theorem \ref{thm:packing-floor} is a statement about \emph{every}
configuration, so it combines with Theorem \ref{thm:main} to pin the optimal
constant between two positive finite bounds, uniformly in $N$:
\[
\frac{\beta^2a_{\Xsurf}^2}{8}\leq\Theta_{N,\beta}(\Xsurf)-1
\leq \exp(\beta C_{\Xsurf,\pmeas})-1.
\]
The quadratic refinement is stated in \Cref{cor:quadratic-bergman}.
The theorem does \emph{not} identify the optimal
constant, decide whether $\Theta_{N,\beta}$ converges as $N\to\infty$, or
describe the minimizing configurations; those questions are the subject of
\cite{PaperB}.
\end{remark}

The bounds above do not identify the sharp floor or ceiling.  For clarity, let
\[
 q_N:=\min_{\Wset\in \Xsurf^N}\sup_{z\in \Xsurf}V_{\Wset}(z).
\]
The minimum exists: by \eqref{eq:green-joint-local},
$Y_{\Wset}(z)=\prod_{j=1}^N \e^{U_{\Xsurf}(z,w_j)}$ is continuous on
$\Xsurf\times \Xsurf^N$,
including collisions.  Hence $\Wset\mapsto\max_zY_{\Wset}(z)$ is continuous
and positive
on compact $\Xsurf^N$, so its logarithm attains its minimum.
The largest admissible constant in \eqref{eq:packing-floor} is
$a_*:=\inf_{N\geq1}\inf_{\Wset\in \Xsurf^N}\|V_{\Wset}\|_{L^1}$, since
$\|V_{\Wset}\|_{L^1}\leq2\sup_{\Xsurf}V_{\Wset}$.
The optimal universal floor for the supremum, its asymptotic floor, and the
best uniform ceiling are, respectively,
\[
 q_*:=\inf_{N\geq1}q_N,\qquad
 q_{\infty}:=\liminf_{N\to\infty}q_N,\qquad
 C_*:=\sup_{N\geq1}q_N.
\]

This leads to the following question concerning the optimal packing floor.

\begin{problem}\label{prob:optimal-floor}
Determine $a_*$, $q_*$, and $C_*$, and decide whether $q_N$ converges as
$N\to\infty$.  Determine the relation of $q_{\infty}$ to the finite-degree
extrema.  The two-point case on $\Xsurf=\bP^1$ is computed in \cite{PaperB}, but
does not identify these extrema over all degrees.
\end{problem}

\subsection{Asymptotic packing and microscopic geometry}

The possible convergence of $q_N$ has a natural geometric interpretation.
At the microscopic scale $N^{-1/2}$, the surface is approximately Euclidean.
This suggests that extremal configurations might locally resemble the
triangular lattice, with hexagonal Voronoi cells.  Such a pattern generally
requires global adjustments to accommodate the topology and geometry of
$\Xsurf$.  It is conceivable that the effect of these adjustments on the
optimal Green-potential supremum becomes negligible as $N\to\infty$, and
that a limiting constant, if it exists, is determined by a planar
logarithmic-potential problem.  Both convergence and the proposed
hexagonal microscopic structure remain open.

\section{Uniformity in families and degeneration}\label{sec:moduli}

\subsection{Bounded-geometry families}

The constants of Theorem \ref{thm:main} depend on the fixed data
$(\Xsurf,\pmeas,\geo,L,\heo)$, and we have said nothing about how.
For applications
to families of curves --- and for any statement near the boundary of moduli
space --- that dependence is the whole question.

This leads to the following question concerning uniformity in moduli.

\begin{problem}\label{prob:uniform-moduli}
For fixed genus $\mathfrak g>1$, determine how the constant in
Theorem \ref{thm:main}
depends on the geometry of $\Xsurf$, especially near the boundary of moduli
space.
\end{problem}

The proof gives uniformity on bounded-geometry families provided the
quantitative transversality constants are uniform.  More precisely, consider
a $C^\infty$-precompact family of triples $(\Xsurf,\pmeas,\geo)$ with
smoothly chosen degree-one polarized data $(L,\heo)$.  Assume a uniform
positive lower bound for the curvature, a uniform injectivity-radius lower
bound, and uniform bounds for the derivatives used in \Cref{app:rescaling}.
Assume in addition that there are common constants $N_0,M,\eta>0$ such that,
on each member and for every $N\geq N_0$, a holomorphic section satisfies
$\|s_N\|_{C^2_N}\leq M$ and
$|s_N|+|\nabla s_N|_{\geo_N}\geq\eta$.  The local zero lemma, the measure
comparison, and \Cref{prop:transversality-to-packing} then give a common
packing constant for $N\geq N_0$.  A common upper bound for the normalized
Green kernels, available for this precompact bounded-geometry family,
handles the finitely many smaller degrees.  Thus \Cref{thm:main} is uniform
under these hypotheses.  The fixed-surface transversality theorem quoted
above does not by itself establish the additional parameter-uniform
transversality hypothesis.

\subsection{Degeneration and collar geometry}

Approaching the boundary of moduli space removes the injectivity-radius
bound.  A plausible strategy is to split the surface into thick and thin
parts by the standard hyperbolic collar lemma. One would need uniform
transversality estimates on the thick part. Each thin component is
a long cylinder, where the Green kernel and the relevant sections can be
compared with their flat-cylinder models and then transported by the collar
coordinates.  The remaining task is to match the thick and thin
constructions while obtaining tail and coloring estimates whose loss is
explicit in the collar length.  This should also reveal the rate at which
the constant in Theorem \ref{thm:main} can diverge near the
Deligne--Mumford boundary.

Route B suggests a complementary thin-part model.  On a cylinder of length
$T$, separate variables into Fourier modes around the circle and slowly
varying modes along the axis.  A design that cancels the modes below the
local cutoff can be compared with an evenly spaced product grid, while the
remaining Green kernel is estimated from the explicit cylindrical spectral
series.  The technical issue is not the model cylinder but the interface:
one must match its approximate quadrature rule to a thick-part design
without losing equal weights or local counting.  Tracking this loss in
$T$ would turn the qualitative moduli question into a candidate upper
growth rate for the packing constant.

\appendix

\section{Rescaled holomorphic coordinates and frames}
\label{app:rescaling}

\subsection{Uniform coordinates and the comparison estimates}

For completeness, we justify the bounded-geometry statement used in
Section \ref{sec:zero-geometry}.  The point is that the connection on
$L^N$ looks uniform, rather than increasingly curved, after distances are
multiplied by $\sqrt N$.

Fix a finite holomorphic atlas on $\Xsurf$.  Around any point $x$, after
shrinking the charts uniformly, choose a coordinate $z$ with $z(x)=0$ in which
\[
 \geo=\lambda_x(z)|dz|^2,
\]
where $\lambda_x$ and its derivatives through any prescribed finite order
are bounded above uniformly, and $\lambda_x$ is bounded below uniformly by
a positive constant.  This is possible by compactness.

Choose a local holomorphic frame $e_x$ for $L$ and write
\[
 |e_x(z)|_{\heo}=\e^{-\phi_x(z)/2}.
\]
Multiplying $e_x$ by a nowhere-vanishing holomorphic function changes
$\phi_x$ by the real part of a holomorphic function.  We may therefore
normalize
\[
 \phi_x(0)=0,
 \qquad
 \partial\phi_x(0)=0.
\]
The curvature identity $c_1(L,\heo)=\pmeas$ implies that the mixed quadratic
term of $\phi_x$ is positive: in these coordinates
$\partial\bar\partial\phi_x(0)$ is a positive multiple of
$dz\wedge d\bar z$, with multiplier bounded above and below because $\pmeas$
is smooth and positive.  Taylor's theorem, uniformly in the finite atlas,
gives
\begin{equation}\label{eq:weight-taylor}
 \phi_x(z)=q_x|z|^2+\operatorname{Re}(a_xz^2)+O(|z|^3),
\end{equation}
with $q_x$ bounded above and below by positive constants and all remainders
uniform in $x$.  The pure holomorphic quadratic term may also be removed by
a holomorphic change of frame, but this is not necessary: it changes only
bounded connection coefficients, not the uniform ellipticity supplied by
$q_x|z|^2$.

Introduce the rescaled coordinate
\[
 u=\sqrt N\,z.
\]
For every fixed $R$, the disk $|u|<R$ corresponds to
$|z|<R/\sqrt N$ and hence lies in the chosen chart for all sufficiently
large $N$, uniformly in the center $x$.
Then
\begin{equation}\label{eq:rescaled-metric}
\geo_N=N\geo=\lambda_x(u/\sqrt N)|du|^2,
\end{equation}
so on every fixed disk $|u|<R$, the coefficients of $\geo_N$ and their
$u$-derivatives are bounded uniformly in $x$ and $N$, and the metric is
uniformly equivalent to the Euclidean metric.  The weight of the tensor
frame $e_x^N$ is
\begin{equation}\label{eq:rescaled-weight}
 N\phi_x(u/\sqrt N)
 =q_x|u|^2+\operatorname{Re}(a_xu^2)+O_R(N^{-1/2}).
\end{equation}
The center normalizations have two consequences used later.  In the tensor
frame $e_x^N$ one has $|e_x^N(0)|_{\heo^N}=1$, and the Chern connection
one-form in the rescaled $u$-coordinate vanishes at $u=0$. Throughout the
paper we use the particular bounded-geometry family
described above, including its uniform derivative bounds, and not merely these
two normalizations at the center.
The same estimate holds after taking any fixed number of $u$-derivatives.
Consequently the coefficients of the Chern connection of $(L^N,\heo^N)$,
expressed in this rescaled holomorphic frame, and their first derivatives
are bounded uniformly on $|u|<R$.

If $s_N=f_{x,N}(u)e_x^N$ is holomorphic, then $f_{x,N}$ is an ordinary
holomorphic function. The connection coefficients are bounded. Hence
\eqref{eq:rescaled-weight}
supplies constants $C_1,C_2$, independent of $x$ and $N$, such that on
a slightly smaller fixed disk
\begin{align}
 C_1^{-1}\bigl(|f_{x,N}|+|f_{x,N}'|\bigr)
 &\le |s_N|+|\nabla s_N|_{\geo_N}
 \le C_1\bigl(|f_{x,N}|+|f_{x,N}'|\bigr),
 \label{eq:jet-comparison}\\
 \sup\bigl(|f_{x,N}|+|f_{x,N}'|+|f_{x,N}''|\bigr)
 &\le C_2\|s_N\|_{C^2_N}.
 \label{eq:C2-comparison}
\end{align}
For the first inequality, write the Chern derivative as
$\nabla^{1,0}f=df+Af$, where $A$ is uniformly bounded, and solve for $df$;
the metric and frame norms are uniformly comparable on the fixed disk.
Differentiating once more gives \eqref{eq:C2-comparison}.  Thus
\eqref{eq:our-uniform-estimates} implies the hypotheses of Lemma
\ref{lem:local-zero} with constants independent of the center and of $N$.
The harmless shrink of the disk keeps derivatives of the chosen frame and
connection away from the chart boundary; all radii used later are fixed
before $N$ is chosen.

Finally, since $\diff\pmeas$ is a smooth positive form, in the original
coordinate it has the form $m_x(z)dA(z)$ with $m_x$ uniformly bounded
above and below.
Under $u=\sqrt N z$ this becomes
\begin{equation}\label{eq:measure-scaling}
\diff\pmeas=\frac1N m_x(u/\sqrt N)\diff A(u),
\end{equation}
which is the scaling used in the proof of Theorem \ref{thm:log-L1}.

This rescaling is the only place where the two routes meet. Route A uses it to
make the curvature of $L^N$ and the first derivative of a
transverse section both of order one, so that a single local model governs
every zero.  Route B uses the same scale spectrally, as the point where
Weyl's law and the spacing of $N$ points agree. That coincidence is why
$\Lambda\asymp\sqrt N$ is forced in one argument
and $\geo_N=N\geo$ in the other.

\subsection{Gaussian-normalized coordinates for the direct construction}\label{direct:app:rescaling}

For completeness, we justify the bounded-geometry statement used throughout
\Cref{direct:sec:peaks,direct:sec:coloring}. The point is that the connection on $L^N$ looks
uniform, rather than increasingly curved, after distances are multiplied by
$\sqrt N$.

Fix a finite holomorphic atlas on $\Xsurf$. Around any point $x$, after shrinking
the charts uniformly, choose a coordinate $z$ with $z(x)=0$ in which
\[
 \geo=\lambda_x(z)|dz|^2,
\]
where $\lambda_x$ and its derivatives through any prescribed finite order
are bounded above uniformly, and $\lambda_x$ is bounded below uniformly by
a positive constant. Compactness makes these bounds uniform.

Choose a local holomorphic frame $e_x$ for $L$ and write
\[
 |e_x(z)|_{\heo}=e^{-\phi_x(z)/2}.
\]
Multiplying $e_x$ by a nowhere-vanishing holomorphic function changes
$\phi_x$ by the real part of a holomorphic function. We may therefore
normalize
\[
 \phi_x(0)=0,
 \qquad
 \partial\phi_x(0)=0.
\]
The curvature identity $c_1(L,\heo)=\pmeas$ implies that the mixed quadratic term
of $\phi_x$ is positive: in these coordinates
$\partial\bar\partial\phi_x(0)$ is a positive multiple of
$dz\wedge d\bar z$, with a multiplier bounded above and below because $\pmeas$
is smooth and positive. Taylor's theorem, uniformly in the finite atlas,
gives
\begin{equation}\label{direct:eq:weight-taylor}
 \phi_x(z)=q_x|z|^2+\operatorname{Re}(a_xz^2)+O(|z|^3),
\end{equation}
where $q_x$ is bounded above and below by positive constants and the
remainders are uniform in $x$. Multiply the frame by the exponential of a
holomorphic quadratic polynomial to remove $\operatorname{Re}(a_xz^2)$, and
rescale $z$ by a uniformly controlled positive factor to arrange $q_x=1$.
These changes preserve all uniform chart and metric bounds. Thus,
\begin{equation}\label{direct:eq:normalized-weight-taylor}
 \phi_x(z)=|z|^2+O(|z|^3)
\end{equation}
uniformly in the center. This normalization is the one used in
\eqref{direct:eq:appendix-BF-weight}; removing the pure quadratic term makes the
Gaussian comparison direction-independent.

Introduce the rescaled coordinate
\[
 u=\sqrt N\,z.
\]
For every fixed $R$, the disk $|u|<R$ corresponds to
$|z|<R/\sqrt N$ and hence lies in the chosen chart for all sufficiently
large $N$, uniformly in the center $x$.
Then
\begin{equation}\label{direct:eq:rescaled-metric}
 \geo_N=N\geo=\lambda_x(u/\sqrt N)|du|^2,
\end{equation}
so on every fixed disk $|u|<R$, the coefficients of $\geo_N$ and their
$u$-derivatives are bounded uniformly in $x$ and $N$, and the metric is
uniformly equivalent to the Euclidean metric. The weight of the tensor
frame $e_x^N$ is
\begin{equation}\label{direct:eq:rescaled-weight}
 N\phi_x(u/\sqrt N)
 =|u|^2+O_R(N^{-1/2}).
\end{equation}
The center normalizations have two consequences used later. In the tensor
frame $e_x^N$ one has $|e_x^N(0)|_{\heo^N}=1$, and the Chern connection
one-form in the rescaled $u$-coordinate vanishes at $u=0$. Away from the center, the uniform derivative bounds of the chosen
family are also needed. The remainder in~\eqref{direct:eq:rescaled-weight} satisfies
the same bound after any fixed number of $u$-derivatives. Consequently the
coefficients of the Chern connection of $(L^N,\heo^N)$, expressed in this
rescaled holomorphic frame, and their derivatives through any fixed order
are bounded uniformly on $|u|<R$.

If $s_N=f_{x,N}(u)e_x^N$ is holomorphic, then $f_{x,N}$ is an ordinary
holomorphic function. Since the connection coefficients are bounded,
\eqref{direct:eq:rescaled-weight}
supplies constants $C_1,C_2$, independent of $x$ and $N$, such that on a slightly
smaller fixed disk
\begin{align}
 C_1^{-1}\bigl(|f_{x,N}|+|f_{x,N}'|\bigr)
 &\le |s_N|+|\nabla s_N|_{\geo_N}
 \le C_1\bigl(|f_{x,N}|+|f_{x,N}'|\bigr),
 \label{direct:eq:jet-comparison}\\
 \sup\bigl(|f_{x,N}|+|f_{x,N}'|+|f_{x,N}''|\bigr)
 &\le C_2\|s_N\|_{C^2_N}.
 \label{direct:eq:C2-comparison}
\end{align}
For the first inequality, write the Chern derivative as
$\nabla^{1,0}f=df+Af$, where $A$ is uniformly bounded, and solve for $df$;
the metric and frame norms are uniformly comparable on the fixed disk.
Differentiating once more gives~\eqref{direct:eq:C2-comparison}. These comparisons
pass between the invariant norms
\eqref{direct:eq:Ck-rescaled}, \eqref{direct:eq:eta-transverse} and the scalar quantities to
which \Cref{direct:prop:estimated-sard} is applied, with constants independent of the
center and of $N$.
The harmless shrink of the disk keeps derivatives of the chosen frame and
connection away from the chart boundary; all radii used later are fixed
before $N$ is chosen.

Finally, since $\diff\pmeas$ is a smooth positive form, in the original coordinate
it has the form $m_x(z)dA(z)$ with $m_x$ uniformly bounded above and below.
Under $u=\sqrt N z$ this becomes
\begin{equation}\label{direct:eq:measure-scaling}
 \diff\pmeas=\frac1N m_x(u/\sqrt N)dA(u),
\end{equation}
which is the measure scaling used in
\eqref{direct:eq:appendix-BF-weight} and in the Cauchy estimates of
\Cref{direct:sec:peaks}.

The construction exploits two features of this rescaling. First, the tensor
frame has Gaussian weight,
$N\phi_x(u/\sqrt N)=|u|^2+O(N^{-1/2})$ on every fixed disk; this makes the
cutoff models of~\Cref{direct:sec:peaks} nearly holomorphic and their $\bar\partial$
defects exponentially small. Second, a fixed rescaled disk carries $\pmeas$-mass
of order $N^{-1}$, which accounts for the factor $N$ in every pointwise Cauchy
estimate. Both features are uniform in the center,
which is why one $N_0$ serves all the colors at once.

\section*{Acknowledgements}
Hedenmalm acknowledges support from the Ministry of Science and Higher
Education of the Russian Federation under agreement 075-15-2025-013, as well
as the hospitality of BIMSA.
S.-Y.~Xie acknowledges partial support from the National Key R\&D Program
of China (grant nos.~2023YFA1010500 and~2021YFA1003100), the National Natural
Science Foundation of China (grant nos.~12288201 and~12471081), and the
Xiaomi Young Talents Program.

\section*{Declaration on the use of artificial intelligence}
The authors formulated the problems and supplied the mathematical ideas.
GPT-5.6-sol was used to assist mathematical exploration and manuscript
preparation. The authors checked all mathematical arguments and take full
responsibility for the contents.

\end{document}